\documentclass[a4paper,12pt]{amsart}
\usepackage[T1]{fontenc}
\usepackage[centering,margin=2.3cm]{geometry}
\usepackage[latin1]{inputenc}
\usepackage{amsthm, amsmath}   
\usepackage{latexsym,amssymb}
\usepackage{algorithmic}
\usepackage{amssymb}
\usepackage{graphicx,subfigure}
\usepackage{times}
\usepackage{enumerate}
\usepackage[usenames,dvipsnames]{pstricks}
\usepackage{epsfig}
\usepackage{pst-node}
\usepackage{pst-grad} 
\usepackage{pst-plot} 
\usepackage{pst-3dplot}
\usepackage{color}

\newcommand{\N}{\mathbb{N}}

\newcommand{\Z}{\mathbb{Z}}

\newcommand{\mS}{\mathcal{S}}

\newcommand{\lS}{\leq_{\mS}}

\theoremstyle{plain}
\newtheorem{theorem}{Theorem}[section]
\newtheorem{lemma}[theorem]{Lemma}
\newtheorem{coro}[theorem]{Corollary}
\newtheorem{prop}[theorem]{Proposition}
\newtheorem{rema}[theorem]{Remark}
\newtheorem{exam}[theorem]{Example}

 \newtheorem{conjecture}[theorem]{Conjecture}

\title[Chomp on numerical semigroups]{Chomp on numerical semigroups}
\author[I. Garc\'{i}a-Marco]{Ignacio Garc\'{i}a-Marco}
\address{Facultad de Ciencias, Universidad de La Laguna. 38200 La Laguna, Spain.}
\email{iggarcia@ull.es}

\author[K. Knauer]{Kolja Knauer}
\address{Aix Marseille Univ, Universit\'e de Toulon, CNRS, LIS, Marseille, France.}
\email{kolja.knauer@lis-lab.fr}


\begin{document}
\begin{abstract}
 We consider the two-player game chomp on posets associated to numerical semigroups and show that the analysis of strategies for chomp is strongly related to classical properties of semigroups.
We characterize which player has a winning-strategy for symmetric semigroups, semigroups of maximal embedding dimension and several families of  numerical  semigroups generated by arithmetic sequences. Furthermore, we show that which player wins on a given numerical semigroup is a decidable question. Finally, we extend several of our results to the more general setting of subsemigroups of $\N \times T$, where $T$ is a finite abelian group.
\end{abstract}

\maketitle

\section{Introduction}\label{introduction}

Let $P$ be a partially ordered set with a global minimum $0$. In the \emph{game of chomp} on $P$ (also known as \emph{poset game}), two players $A$ and $B$ alternatingly pick an element of $P$. Whoever is forced to pick $0$ loses the game. A move consists of picking an element $x \in P$ and removing its {\it up-set}, that is,
all the elements that are larger or equal to $x$. The type of questions one wants to answer is: for a given $P$, has either of the players a winning strategy? And, if yes, can a strategy be devised explicitly? A crucial, easy, and well-known observation with respect to these questions is the following one.

\begin{rema}\label{ganaAmax} If $P$ has a global maximum $1$ and $1 \neq 0$ (that is, $P \neq \{0\})$, then player $A$ has a winning strategy. This can be proved with an easy strategy stealing argument.
Indeed, if $A$ starting with $1$ cannot be extended to a winning strategy, there is a devastating reply $x \in P$. But in this case, $A$ wins starting with $x$. 
\end{rema}

One of the most well-known and probably oldest games that is an instance of chomp is Nim~\cite{Bouton}, where $P$ consists of a disjoint union of chains plus a global minimum. The first formulation in terms of posets is due to Schuh~\cite{Schuh}, where the poset is that of all divisors of a fixed number $N$, with $x$ below $y$ when $y|x$. A popular special case of this is the chocolate-bar-game introduced by Gale~\cite{Gale}. The \emph{$a \times b$-chocolate-bar-game} coincides with chomp on $P := \{(n,m) \mid 0 \leq n < a$, $0 \leq m < b\}$ partially 
ordered by $(n,m) \leq (n',m')$ if and only if $n \leq n'$ and $m \leq m'$. Since $(P,\leq)$ has $(a-1,b-1)$ as global maximum, the strategy stealing argument of Remark~\ref{ganaAmax} yields that player $A$ has a winning strategy for chomp on $P$. Explicit strategies for player $A$ are known only in very particular cases including 
when $a = b$ or $a \leq 2$. The problem remains open even in the $3 \times b$ case. There is a rich body of research on chomp with respect to different classes of posets. For more information on the game and its history, we refer to~\cite{web,fr}.

In this work, we investigate chomp on a family of posets arising from additive semigroups of natural numbers instead of their multiplication as in Schuh's game. More precisely, let $\mS = \langle a_1,\ldots,a_n \rangle$ denote the semigroup
generated by positive integers $a_1, \ldots,a_n$, that is,
$$
\mS = \langle a_1,\ldots,a_n \rangle = \{x_1 a_1 + \cdots + x_n a_n \, \vert \, x_1,\ldots,x_n \in \N\}.
$$ 
The semigroup $\mS$ induces on itself a poset structure $(\mS, \lS)$ whose partial
order $\lS$ is defined by \begin{center} $x \lS y \Longleftrightarrow y - x \in
\mS$. \end{center} 
This poset structure on $\mS$ has been considered in \cite{fgh,Kunz} to study algebraic properties of its corresponding semigroup algebra, and in \cite{cgpr,cr,ded} to study its M\"obius function. We observe that there is no loss of generality in assuming that $a_1,\ldots,a_n$ are relatively prime. Indeed,
setting $d := \gcd(a_1,\ldots,a_n)$ and $\mS' := \langle a_1/d,\ldots,a_n/d \rangle$, then $(\mS,\lS)$ and $(\mS', \leq_{\mS'})$ are isomorphic posets. From
now on we will assume that $\mS$ is a {\it numerical semigroup}, that is, it is generated by relatively prime integers. 

One of the peculiarities of chomp on posets coming from numerical semigroups is that after player $A$'s first move, the remaining poset is finite (see Remark~\ref{aperychomp}).
In particular, this implies that every game is finite and either $A$ or $B$ has a winning strategy.

\subsubsection*{Results and structure of the paper:} In Section~\ref{sec:sym} we introduce some elements of the theory of numerical semigroups. In particular, we relate the Ap\'ery sets to the first moves of chomp and show that $B$ has a winning strategy in all symmetric semigroups (Theorem~\ref{simetricoBgana}). Section~\ref{sec:max} is devoted to numerical semigroups of maximal embedding dimension. We characterize those semigroups in the class for which $A$ has a winning strategy (Theorem~\ref{medchomp}) and describe an explicit winning strategy. In Section~\ref{sec:arithmetic} we discuss numerical semigroups that are generated by generalized arithmetic sequences. We characterize, when the smallest generator is a winning move for $A$ (Proposition~\ref{juegaaaritm}) and classify those semigroups in the class, that have three generators and admit a winning strategy for $A$ (Theorem~\ref{genarit3gen}). In Section~\ref{sec:intervals} we consider semigroups generated by intervals. We characterize when $A$ has a winning strategy for $\mS = \langle a, a+1,\ldots, 2a-3\rangle$ with $a \geq 4$ (Proposition~\ref{tipo2consec}) and $\mS = \langle 3k, 3k+1,\ldots, 4k \rangle$ with $k$ odd (Proposition~\ref{juegaa+1aritm}). In Section~\ref{sec:bound} we show, that if $A$ has a winning strategy on $\mS$, then the smallest winning first move is bounded by a function of the number of gaps and the Frobenius number of $\mS$ (Theorem~\ref{boundonstrategy}). In particular, which player wins on a given semigroup is a decidable question. Finally, in Section~\ref{sec:torsion}, we investigate to what extent the game of chomp can be generalized to other algebraically defined posets. In particular, we extend the decidability result  (Theorem~\ref{boundonstrategylattice}) as well as the result on symmetric semigroups (Theorem~\ref{simetricoBganalattice}) to the more general setting of numerical semigroups with torsion, that is, subsemigroups of $\N \times T$, where $T$ is a finite abelian group. We conclude the paper with some questions in Section~\ref{sec:conclude}.

\section{Symmetric numerical semigroups}\label{sec:sym}

One of the most studied families of numerical semigroups is the family of symmetric numerical semigroups. This family contains the family of complete intersection semigroups and is an important subfamily of irreducible semigroups (see, for example,~\cite{RR,Rosa}). Symmetric numerical semigroups turn out to be also interesting in the study of affine monomial curves; in particular, Kunz~\cite{Kunz} proved that a numerical semigroup $\mS$ is symmetric if and only if its corresponding  one-dimensional semigroup algebra $K[\mS] = K[t^s\, \vert \, s \in \mS]$ is Gorenstein.

The goal of this short section is to prove that whenever $\mS$ is symmetric, then player $B$ has a winning strategy. One of the key points to obtain this result will be Remark~\ref{ganaAmax}.

We will now recall some basic results and definitions about numerical semigroups, for a detailed exposition on this topic we refer the
reader to~\cite{Alibro,RGlibro}. For a numerical semigroup $\mS$, 
the \textit{Frobenius number} $g=g(\mS)$ is the largest integer not in $\mS$. The semigroup $\mS$ is {\it symmetric} if 
\[
\mS\cup(g-\mS)=\Z,
\]   
where $g-\mS=\{g-s\mid s\in \mS\}$. In other words, $\mS$ is symmetric if and only if for every $x \in \Z$ either $x \in \mS$ or $g-x \in \mS$.

For a semigroup $\mS$ we set $\mathit{PF} = \{x \notin \mS \, \vert \, x+s \in \mS$ for all $s \in \mS \setminus \{0\}\}.$ The elements
of $\mathit{PF}$ are usually called {\it pseudo-Frobenius numbers} and the number of elements
of $\mathit{PF}$ is called the {\it type} of $\mS$ and denoted by ${\rm type}(\mS)$ (it coincides with
the Cohen-Macaulay type of the corresponding semigroup algebra $K[t^s \, \vert \, s \in \mS]$, where $K$ is any field). We notice that $g$ is always a pseudo-Frobenius number. 
Moreover, by~\cite[Proposition 2]{fgh}
$\mS$ is symmetric if and only if $\mathit{PF} = \{g\}$, or equivalently, when the type of $\mS$ is $1$.

Given $a \in \mS$, the {\it Ap\'ery set} of $\mS$ with respect to $a$ is defined as
$${\rm Ap}(\mS, a) = \{s \in \mS\, \vert \, s-a \notin \mS\}.$$ 
This set, which is a complete set of residues modulo $a$, contains relevant information about the semigroup. The following result illustrates how to obtain the set of pseudo-Frobenius numbers from an Ap\'ery set.

\begin{prop}~\cite[Proposition 7]{fgh} Let $\mS$ be a numerical semigroup and $a \in \mS$. Then, the following conditions are equivalent for any $x \in \Z$:  
\begin{enumerate}[$\bullet$]
\item $x$ is a pseudo-Frobenius number of $\mS$, and   
\item $x + a$ is a maximal element
of ${\rm Ap}(\mS,a)$ with respect to the partial order $\leq_{\mS}$.
\end{enumerate}\end{prop}

In particular, the following corollary holds. 

\begin{coro}\label{typeinvariant} Let $\mS$ be a numerical semigroup and $a \in \mS$. The number of maximal elements of ${\rm Ap}(\mS,a)$ with respect to $\leq_{\mS}$ does not depend on the $a \in \mS$ chosen and coincides with ${\rm type}(\mS)$.
\end{coro}

Going back to chomp, the following remark will be essential in the sequel and establishes a nice connection between chomp and the theory of numerical semigroups.

\begin{rema}\label{aperychomp} Whenever player $A$'s first move is $a \in \mS$, then the remaining poset is exactly
${\rm Ap}(S,a)$.
\end{rema}

This remark together with Corollary~\ref{typeinvariant} and Remark~\ref{ganaAmax} yields the following.

\begin{theorem}\label{simetricoBgana}If $\mS$ is a symmetric numerical semigroup, then player $B$ has a winning strategy for chomp on $\mS$.
\end{theorem}
\begin{proof}No matter the first move of $A$, the remaining poset has a global maximum because ${\rm type}(\mS) = 1$. Hence, Remark~\ref{ganaAmax} yields the result.
 \end{proof}

We observe that the proof of Theorem~\ref{simetricoBgana} is not constructive. The task of determining a winning strategy for $B$ in this context seems a
very difficult task. Indeed, we will argue now that it is at least as difficult as providing a winning strategy for $A$ in the $a \times b$-chocolate-bar-game, which, as mentioned in the introduction, remains an unsolved problem even for $a=3$.
%
Namely, if $\mS$ is symmetric, to describe a winning strategy for $B$, one has to exhibit
a good answer for $B$ for every first move of $A$. It is well known (and easy to check) that every
two-generated numerical semigroup is symmetric. Let us consider $\mS = \langle a, b \rangle$ with $a,b$ relatively prime and assume that $A$ starts picking $ac$ with $1 \leq c \leq b$. The remaining poset is
 $${\rm Ap}(\mS, ac) = \{\lambda a + \mu b  \, \vert \, 0 \leq \lambda < c, 0 \leq \mu < a\},$$
 and, for every  $0 \leq \lambda, \lambda' < c$, $0 \leq \mu, \mu' < a$, we have that $\lambda a + \mu c  \leq_{\mS} \lambda' a + \mu' c$ if and only if
 $\lambda \leq \lambda',\ \mu \leq \mu'$. Hence, the remaining poset is isomorphic to the $a \times c$ chocolate bar. Therefore,
providing an explicit winning strategy for $B$ in this poset is equivalent to finding winning strategies for $A$ in chomp on the $a \times c$ chocolate bar for every $1 \leq c \leq b$.

\section{Numerical semigroups of maximal embedding dimension}\label{sec:max}

Given a numerical semigroup $\mS$, its smallest nonzero element is usually called its {\it multiplicity} and denoted by $m(\mS)$. It turns out that every numerical semigroup has a unique minimal generating set, the size of this set is the {\it embedding dimension of $\mS$} and is denoted by $e(\mS)$. It is easy to prove that $m(\mS) \geq e(\mS)$.
Whenever $m(\mS) = e(\mS)$ we say that $\mS$ has {\it maximal embedding dimension}.  
In this section we study chomp on posets coming from numerical semigroups with maximal embedding dimension.

These semigroups have been studied by several authors, see, for instance,~\cite{BDF,MED}. There exist  several characterizations for this family of semigroups,
in particular we use the following ones, which can be found or easily deduced from the ones in~\cite[Chapter 2]{RGlibro}.

\begin{prop}\label{medpropiedades}Let $\mS$ be a numerical semigroup and let $a_1 < \cdots < a_n$ be its unique minimal system of generators. Then the following properties are equivalent.
\begin{itemize}
\item[(a)] $\mS$ has maximal embedding dimension,
\item[(b)] ${\rm Ap}(\mS,a_1) = \{0,a_2,\ldots,a_n\},$
\item[(c)] ${\rm type}(\mS) = a_1 - 1$, 
\item[(d)] $x+y-a_1 \in \mS$ for all $x,y \in \mS \setminus \{0\}$,
\item[(e)] every element $b \in \mS \setminus \{0\}$ can be uniquely written 
as $b = \lambda a_1 + a_i$ with $\lambda \in \N$ and $i \in \{1,\ldots,n\}$.
\end{itemize}
\end{prop}

For this family of posets we do not only characterize which player has a winning a strategy but also we provide an explicit winning strategy in each case.

We start with a result about Ap\'ery sets that we will need in the sequel. For every numerical semigroup, whenever $a,b \in \mS$ and $a \leq_{\mS} b$, it is easy to check
that ${\rm Ap}(\mS,a) \subset {\rm Ap}(\mS,b)$. The following result furthermore describes the difference between these two sets.

\begin{prop}\cite[Lemma 2]{fghl}\label{aperyunion} Let $\mS$ be a numerical semigroup and $a,b \in \mS$ such that $a \leq_{\mS} b$. Then,
$${\rm Ap}(\mS,b) = {\rm Ap}(\mS,a) \cup  (a + {\rm  Ap}(\mS,b-a) ),$$
and it is a disjoint union.
\end{prop}

In particular, from this result and Proposition~\ref{medpropiedades}.(b) one easily derives the following result.

\begin{lemma}\label{lemmaMED}Let $\mS = \langle a_1, \ldots, a_m \rangle$ be a maximal embedding dimension numerical semigroup with multiplicity $m = a_1$.
Then, \[ {\rm Ap}(\mS,\lambda m) = \{\mu m, \mu m + a_i \, \vert \, 0 \leq m < \lambda, 2 \leq i \leq m\};\]
and the set of maximal elements of ${\rm Ap}(\mS,\lambda m)$ with respect to $\lS$ is $\{(\lambda-1) m + a_i \, \vert \, 2 \leq i \leq m\}.$ 
\end{lemma}

Now, we can proceed with the proof of the main result of this section.

\begin{theorem}\label{medchomp}
Let $\mS$ be a maximal embedding dimension numerical semigroup with multiplicity $m$. Player $A$ has a winning strategy for chomp on $\mS$ if and only if $m$ is odd.
\end{theorem}
\begin{proof}Let $m = a_1 < a_2 < \cdots < a_{m}$ be the minimal set of generators of $\mS$.

$(\Leftarrow)$ If $m$ is odd and player $A$ picks $m$, then the remaining poset is $Ap(\mS,m) = \{0,a_2,\ldots,a_m\}$ which
has $0$ as minimum and $m-1$ non-comparable (with respect to $\leq_{\mS}$) maximal elements (see Figure~\ref{fig:poset1}). Since $m-1$ is even, it is easy to see that player $B$ loses this game.
 \begin{figure}[htb]
  \centering
  \includegraphics{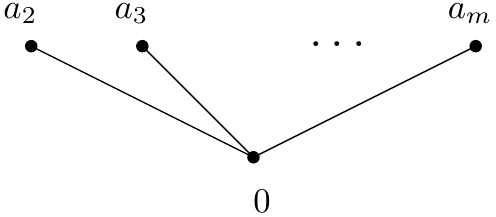}
  \caption{The poset ${\rm Ap}(S, m)$.}
  \label{fig:poset1}
 \end{figure}
 
$(\Rightarrow)$ 
To prove this implication we divide $\mS$ into layers $\{\mS_{\lambda}\}_{\lambda  \in \N}$, being $\mS_0 := \{0\}$ and $\mS_{\lambda} = \{(\lambda-1) m + a_i, \vert \, 1 \leq i \leq m\}$ for $\lambda \geq 1$. By Lemma~\ref{lemmaMED}  we have that  ${\rm Ap}(S, \lambda m) = (\cup_{\mu \leq \lambda} \mS_{\mu}) \setminus \{\lambda m\}$ and has exactly $m-1$ maximal elements, namely those of $\mS_{\lambda} \setminus \{\lambda m\}$.

To prove that $B$ has a winning strategy, we show that whenever $A$'s first move is an element of $\mS_{\lambda}$,  then
$B$ has an answer in $\mS_{\lambda}$ so that:
\begin{itemize}
\item[(a)] the remaining poset $P$ satisfies that $\cup _{\mu < \lambda} \mS_{\mu} \subseteq P \subsetneq \cup _{\mu \leq \lambda} \mS_{\mu}$, and
\item[(b)] $A$ is eventually forced to pick an element from a layer $\mS_{\mu}$ with $\mu < \lambda$. 
\end{itemize}

 We separate two cases. If $A$ picks $x = \lambda m$, then the remaining poset is exactly ${\rm Ap}(S, \lambda m)$. This poset has $m-1$ 
maximal elements which are exactly the elements of $\mS_{\lambda} \setminus \{\lambda m\}$. Since $m-1$ is odd, $B$ can force $A$ to pick an element of a layer $\mu < \lambda$ by just choosing elements from $\mS_{\lambda}$. If $A$ picks $x = (\lambda-1) m + a_i$ with $i \geq 2$, then $\lambda a_1 \in {\rm Ap}(S,x)$. Now $B$ picks $\lambda a_1$ and the remaining poset is exactly ${\rm Ap}(S, \lambda m) \setminus \{x\}$. This poset has $m-2$ elements of $\mS_{\lambda}$ and they are all maximal. Since $m-2$ is even, $B$ can force $A$ to pick an element of a layer $\mu < \lambda$. Iterating this strategy yields a winning strategy for $B$.
  \end{proof}

As we mentioned before, this result does not only characterize who has a winning strategy but also provides explicit winning strategies.

%

\section{Numerical semigroups generated by generalized arithmetic sequences}\label{sec:arithmetic}

A {\emph generalized arithmetic sequence} is set of positive integers of the form $a < ha+d < \cdots < ha+kd$ for some $a, d, k, h \in \Z^+$. Since such a sequence 
generates a numerical semigroup if and only if $\gcd\{a,d\} = 1$, from now on we assume that this is the case.
Several authors have studied semigroups generated by generalized arithmetic sequences (see, for example \cite{EL,Mat2004,Selmer}) as well as their relation with
 monomial curves (see, for example, \cite{BGG,SZ}). This section concerns chomp on semigroups generated by a generalized arithmetic sequence of positive integers. The main result of this section is Theorem~\ref{genarit3gen}, where we characterize which player has a winning strategy for chomp on $\mS$ when $k = 2$.

A first easy observation is that $\{a,ha+d,\ldots,ha+kd\}$ minimally generates $\mS$ if and only if $k < a$; otherwise $\mS = \langle a, ha + d, \ldots, a(h+d) - d \rangle$. Thus, from now on we assume that $k < a$. We observe that a semigroup $\mS$ generated by a generalized arithmetic sequence is of maximal embedding dimension if and only if $k = a-1$. Hence, when $k = a-1$, Theorem~\ref{medchomp} applies here to conclude that $A$ has a winning strategy if and only if $a$ is odd.


We are now going to characterize when $a$ is a winning first move for $A$ in the chomp game; for this purpose we study ${\rm Ap}(\mS,a)$. Although the description of ${\rm Ap}(\mS,a)$ is due to Selmer~\cite{Selmer} (see also~\cite{Mat2004}), here we include a slightly refined version of his result where we also describe the ordering $\lS$ on ${\rm Ap}(\mS,a)$.

\begin{prop}\label{aperyaaritmetica}Let $\mS = \langle a, ha+d, \ldots, ha+kd \rangle$ with $a, k, d, h \in \Z^+,\ k < a$ and $\gcd\{a,d\} = 1$. Then, 
\[
	{\rm Ap}(\mS,a) = \left\{  \left\lceil \frac{i}{k} \right\rceil ha + id \, \vert \, 0 \leq i < a \right\}.
\]
Moreover, if we take $t \equiv (a - 1)\ {\rm mod}\ k$, $t \in \{1,\ldots,k\}$, denote $x_{j,\ell} := j(ha + kd) - k d + \ell d$ for $j,\ell \in \N$, and set

\[ A_j := \left\{  \begin{array}{llll} \{x_{0,k}\} &$ if $& j = 0, \\ \{x_{j,\ell}  \, \vert  \, 1 \leq \ell \leq k\} & $ 
if $& 1 \leq j < \left\lceil \frac{a-1}{k} \right\rceil$,  $ \\ \{ x_{j,\ell} \, \vert \, 1 \leq \ell \leq t\} & $ if $& j =  \left\lceil \frac{a-1}{k} \right\rceil; \end{array} \right. \] 
then we have that ${\rm Ap}(\mS,a)$ is the disjoint union of $A_j$ with $0 \leq j \leq \left\lceil \frac{a-1}{k} \right\rceil $. In addition, if 
$x_{j,\ell} \in A_j$, $x_{j',\ell'} \in A_{j'}$; then,
\[ x_{j,\ell}  <_{\mS} x_{j',\ell'} \Longleftrightarrow  j < j' {\text \ and \ } \ell \geq \ell'. \]
\end{prop}

\begin{figure}[ht]
\begin{center}
\includegraphics[scale=1]{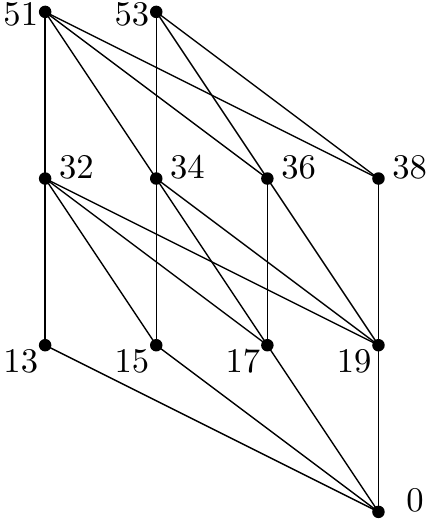}
\caption{The poset ${\rm Ap}(\mS,11)$ with $\mS = \langle 11, 13, 15, 17, 19 \rangle$.}
\end{center}
\label{fig:ap11} 

\end{figure}

In particular, from Proposition~\ref{aperyaaritmetica} one gets that the set of maximal elements of ${\rm Ap}(\mS,a)$ is  $A_j$ with $j = \left\lceil \frac{a-1}{k} \right\rceil$.
According to Corollary~\ref{typeinvariant}, this gives that ${\rm type}(\mS) = t$ being $t \equiv a-1 \ ({\rm mod}\ k)$
and $t \in \{1,\ldots,k\}$. In particular,
$\mS$ is symmetric if and only if $a \equiv 2 \ ({\rm mod}\ k)$ or, equivalently, if $a-2$ is a multiple of $k$ (this result was already known, see Estrada and L\'opez~\cite{EL}, and
Matthews~\cite{Mat2004}). Moreover, if $k$ is even, then ${\rm type}(\mS)$ is even if and only if $a$ is odd.

The following result characterizes when $a$ is a winning first move for chomp on a numerical semigroup generated by a generalized arithmetic sequence.

\begin{prop}\label{juegaaaritm}Let $\mS = \langle a,ha+d,\ldots, ha+kd \rangle$ be a numerical semigroup. Then,
$a \in \mS$ is a winning first move in the chomp game on $(\mS,\lS)$ if and only if $a$ is odd and $k$ is even.
\end{prop}

To prove this result, we will use Proposition~\ref{aperyaaritmetica}, together with the following Lemma (which  is an improvement of 
\cite[Fact 1.5]{fr}):

\begin{lemma}\label{downsetstrat}Let $P$ be a finite poset with a minimum element $0$ and let $\varphi: P \rightarrow P$ be such that 
\begin{itemize} 
\item[(a)] $\varphi \circ \varphi = {\rm id}_P$ (that is, $\varphi$ is an involution),
\item[(b)] if $x \leq \varphi(x)$, then $x = \varphi(x)$,
\item[(c)] if $x \leq y$, then $\varphi(x) \leq \varphi(y)$ or $x \leq \varphi(y)$, and
\item[(d)] the set $F := \{x \in P \, \vert \, \varphi(x) = x\}$ of fixed points of $\varphi$ is a \emph{down-set} of $P$ (that is, $F$ is a subset of $P$ such that if $x \in F$ and $y \leq x \Rightarrow y \in F$).
\end{itemize}
Then, $A$ has a winning strategy on $P$ if and only if $A$ has a winning strategy on $F$. 
\end{lemma}
\begin{proof}Assume that player $A$ has a winning strategy on $F$ and let us exhibit a winning strategy on $P$. Player $A$ starts by picking a winning move in $F$. Now, whenever $B$ picks $x \in P$, we separate two cases. 

{\em Case $x \in F$}. $A$ replies in $F$ following his winning strategy on $F$.

{\em Case $x \notin F$.} $A$ replies $\varphi(x) \notin F$.  

We claim that this gives a strategy for $A$. Indeed, since $F$ is a down-set, picking an element $x \notin F$ at any stage of the game does not alter the remaining elements of $F$. Moreover, conditions (a), (b) and (c) justify that at any state of the game, whenever $B$ picks an element $x \notin F$, the element $\varphi(x)$ is in the remaining poset and, hence, it can be removed by $A$ in the following move. Since the poset is finite, player $B$ cannot pick elements not belonging to $F$ forever and thus, will be forced to play in $F$. Eventually $B$ will be forced to take the minimum because $A$ has a winning strategy on $F$. 

If $A$ does not have a winning strategy on $P$, then $B$ has one (because $P$ is finite). A similar argument to the previous one gives a winning strategy for $B$, completing the proof.
 \end{proof}

\begin{proof}[of Proposition~\ref{juegaaaritm}.] By assuming that A picks $a$, we have that the resulting poset is $({\rm Ap}(\mS,a),\lS)$.
If $k$ is odd, then for $B$ to pick $ha+kd$ is a winning move since the resulting poset is $P = \{0\} \cup \{ha+d,\ldots,ha + (k-1) d\}$ and $P \setminus \{0\}$ is an  antichain  (a set of pairwise incomparable elements) with an even number of elements (see Figure~\ref{fig:antichain}).
 
\begin{figure}[ht]
  \centering
 \includegraphics[scale=1]{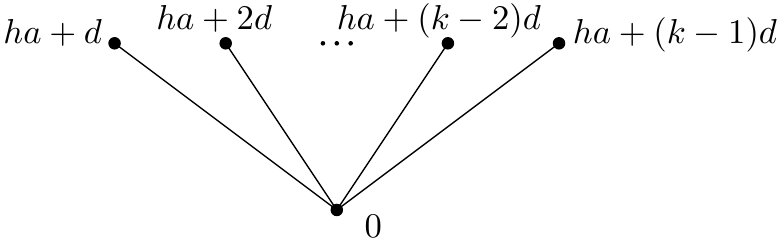}
 \caption{An even antichain with $k-2$ elements and a global minimum.}
 \label{fig:antichain}
\end{figure}  
 
 Suppose now that $k$ is even. Following the notation of Proposition~\ref{aperyaaritmetica}, we first assume that $a$ is odd, that is, $t$ is even, and consider the involution $\varphi: {\rm Ap}(\mS,a) \rightarrow {\rm Ap}(\mS,a)$ defined as $\varphi(x_{j,\ell}) = x_{j,\ell'}$, with
\[ \ell' = \left\{ \begin{array}{llll} 
	k & \text{if } j = 0, \\	
	\ell + 1 & \text{if }\ell\ \text{is odd}, \\
	\ell - 1 & \text{if } \ell \ \text{is even}.
	 \end{array} \right.\]
 A simple argument yields that conditions (a), (b) and (c) in Lemma~\ref{downsetstrat} are satisfied and that the only fixed point of $\varphi$
  is $0$ (see Figure~\ref{fig:involution}). Thus, a direct application of Lemma~\ref{downsetstrat} yields that $A$ has a winning strategy.
  
  \begin{figure}[ht]
    \centering
  \includegraphics[scale=1]{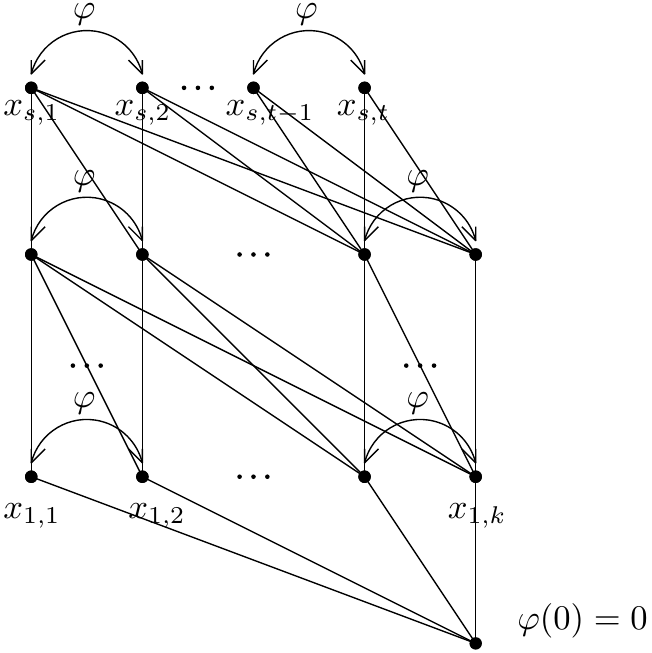}
  \caption{Involution $\varphi: {\rm Ap}(\mS, a) \rightarrow {\rm Ap}(\mS,a)$ in the proof of Proposition~\ref{juegaaaritm} (with $s := \lceil \frac{a-1}{k} \rceil$, $a$ odd and $k$ even).}
  \label{fig:involution}
  \end{figure}
  If $a$ is even, that is, $t$ is odd, we are going to show that $y = x_{s,t}$ with $s = \left\lceil \frac{a-1}{k} \right\rceil$ is a winning move for $B$. Indeed, the same involution $\varphi$
  as before  defined in ${\rm Ap}(S,a) \setminus \{y\}$ proves that $B$ has a winning strategy in this poset. This proves the result.
 \end{proof}

As a consequence of this result, one can completely determine who wins chomp on $\mS$ when $\mS$ is generated by a generalized arithmetic sequence of three elements. 

\begin{theorem}\label{genarit3gen}Let $h, a, d \in \Z^+$ with $a,d$ relatively prime integers and consider $\mS := \langle a, ha+d, ha+2d \rangle$. Player $A$ has a winning strategy for chomp on $\mS$ if and only if $a$ is odd. 
\end{theorem}
\begin{proof}We set $t := {\rm type}(\mS)$. By Proposition~\ref{aperyaaritmetica}, we have that $t \in \{1,2\}$ and $t \equiv a-1\, ({\rm mod}\, 2)$. Hence, 
$\mS$ is symmetric if and only if $a$ is even and, in this case, $B$ has a winning strategy by Theorem~\ref{simetricoBgana}. Whenever $a$ is odd, then Proposition~\ref{juegaaaritm} yields  that $a$ is a winning first move for player $A$. 
 \end{proof}

For three-generated numerical semigroups generated by a generalized arithmetic sequence we have proved that either the multiplicity of the semigroup is a winning first move for $A$, or $B$ has a winning strategy. This is no longer the case for general three-generated numerical semigroups. for instance, an exhaustive computer aided search shows that for $\mS = \langle 6,7,11 \rangle$, the smallest winning first move for $A$ is $25$, and for $\mS = \langle 6, 7, 16 \rangle$ the smallest winning first move for $A$ is 20. It would be interesting to characterize which player has a winning strategy when $\mS$ is three generated.

\section{Numerical semigroups generated by an interval}\label{sec:intervals}

In this section we are going to study chomp on numerical semigroups generated by an interval of positive integers, that is, when $\mS = \langle a,\ldots,a+k \rangle$ for some $a, k \in \Z^+$. These semigroups were studied in detail in~\cite{GRconsec} and form a subfamily of those generated by a generalized arithmetic sequence. Hence, the results obtained in the previous section are also valid in this context. That is, $a$ is a winning first move for $A$ in $\mS = \langle a,\ldots,a+k \rangle$ if and only if $a$ is odd and $k$ is even (Proposition~\ref{juegaaaritm}) and $A$ wins in $\mS = \langle a, a+1 ,a+2 \rangle$ if and only if $a$ is odd (Theorem~\ref{genarit3gen}). Moreover, $\mS = \langle a,\ldots,2a-1 \rangle$ is of maximal embedding dimension, thereby, in this case $A$ has a winning strategy if and only if $a$ is odd (Theorem~\ref{medchomp}). Finally, the semigroups such that $a-2$ is a multiple of $k$, such as $\mS = \langle a,\ldots,2a-2 \rangle$ and $\mS = \langle a, a+1 \rangle$, are symmetric and therefore $B$ has a winning strategy (Theorem~\ref{simetricoBgana}). 

In this section we will extend this list of results to two infinite families. Furthermore, we have implemented an algorithm that 
receives as input a numerical semigroup $\mS$ given by its generators and a positive integer $x$, and tests if player $A$ has a winning strategy with a first move less or equal than $x$. 
Our implementation consists of an exhaustive brute force implementation that checks for all $y \in \mS,\, y \leq x$ if player $A$ has a winning strategy starting with $y$ (we recall that after the first move, the remaining poset is finite). Due to 
the large number of cases to consider, our naive implementation can only handle instances where $x$ is small in a reasonable time. However, the computer search is used for solving one case of Proposition~\ref{tipo2consec}. Moreover, our experiments provide computational evidence for what in in our opinion would be the next feasible questions to attack, concerning chomp on numerical semigroups.

The following result shows that, in semigroups of the form $\mS = \langle 3k, 3k+1,\ldots, 4k \rangle$ for $k$ odd, playing $3k+1$, the second generator, is a winning first move for $A$. 
Note that $3k$, the first generator, is not a winning move  by Proposition~\ref{juegaaaritm}. One member of this family after $A$ picked $3k+1$ is displayed in Figure~\ref{fig:ap16en1520}.

\begin{figure}[ht]
  \centering
\includegraphics[scale=.9]{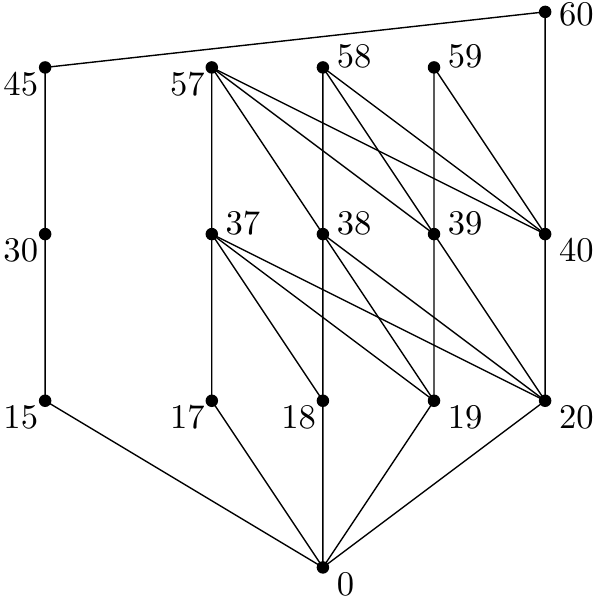}
\caption{${\rm Ap}(\mS,16)$ for $\mS = \langle 15, 16, 17, 18, 19, 20 \rangle$.}\label{fig:ap16en1520}
\end{figure}

\begin{prop}\label{juegaa+1aritm}Let $k\geq 3$ and $\mS = \langle 3k, 3k+1,\ldots, 4k \rangle$. If $k$ is odd, then $3k + 1 \in \mS$ is a winning first move in the chomp game on $(\mS,\lS)$.
\end{prop}

We are going to prove this proposition by means of Lemma~\ref{downsetstrat} together with the following technical lemmas.

\begin{lemma}\label{lemita}Let $\mS = \langle 3k, 3k+1,\ldots, 4k \rangle$. Then, 
\[ 
	{\rm Ap}(\mS,3k+1) = \{0,3k,6k,9k\} \cup \{3k+i, 7k + i, 11k + i \, \vert \, 2 \leq i \leq k\}. 
\]
As a consequence, ${\rm type}(\mS) = k-1$.
\end{lemma}
\begin{proof} We begin by observing that $\mS = \{3k,\ldots,4k\} \cup \{6k,\ldots,8k\} \cup \{x \in \N \, \vert \, x \geq 9k\}$.
Let $A := \{0,3k,6k,9k\} \cup \{3k+i, 7k + i, 11k + i \, \vert \, 2 \leq i \leq k\}$. Since $A \subset \mS$ and
has $3k+1$ elements, to prove that $A = {\rm Ap}(\mS,3k+1)$ it suffices to observe that $x - (3k+1) \notin \mS$ for all $x \in A$. The set of
 maximal elements of $A$ with respect to $\lS$  is $\{11k+i \, \vert \, 2 \leq i \leq k\}$. Thus, by Corollary \ref{typeinvariant} we have that ${\rm type}(\mS) = k-1$.
 \end{proof}

\begin{lemma}\label{poset10elem}Player $B$ has a winning strategy in the poset $(P, \leq)$ with $P = \{0\} \cup \{x_{i,j} \, \vert \, 1 \leq i,j \leq 3\}$ 
and relations induced by 
\[ \begin{array}{llll} 
0 \leq x_{1,j}, & {\text for} &  j = 1,2,3, \\
x_{i,j} \leq x_{i+1,j}, & {\text for} & i = 1,2, &  j = 1,2,3, \\
x_{i,3} \leq x_{i+1,2}, &{\text for} & i = 1,2,  \\
x_{3,1} \leq x_{3,3}.
\end{array} \]  
\end{lemma}
\begin{proof}
We will exhibit a winning answer for any first move of $A$ in this poset (see Figure~\ref{fig:poset10elements}).
\begin{figure}[ht]  \centering
\includegraphics[scale=.9]{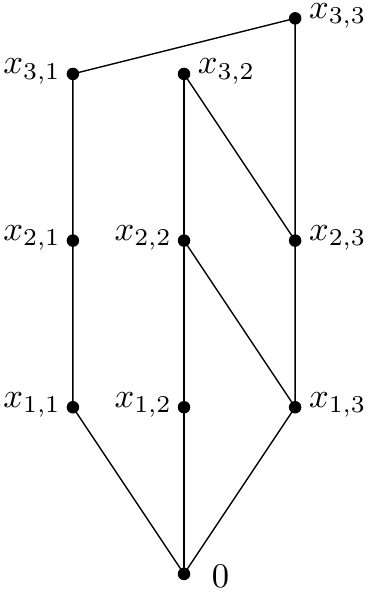}
\caption{The poset of Lemma~\ref{poset10elem}.}\label{fig:poset10elements}
\end{figure}
\begin{itemize}
\item If $A$ picks $x_{1,1}$, then $B$ picks $x_{3,2}$ and vice versa.
\item If $A$ picks $x_{2,1}$, then $B$ picks $x_{1,3}$ and vice versa.
\item If $A$ picks $x_{3,1}$, then $B$ can either pick $x_{1,2}$ or $x_{2,3}$ and vice versa.
\item If $A$ picks $x_{2,2}$, then $B$ picks $x_{3,3}$ and vice versa.
\end{itemize}
Hence, $B$ has a winning strategy.
 \end{proof}

\begin{proof}[of Proposition~\ref{juegaa+1aritm}.] Assume that $k$ is odd. We consider $\varphi$ defined as (see Figure \ref{fig:phi2}). 
\[
	\varphi(x) = \left\{ \begin{array}{lll} 	x, & $ for $ & x \in  \{0, 3k, 6k, 9k, 4k-1,4k, 8k-1,8k, 12k-1,12k\}, \\
											x + 1, & $ for $ & x \in \{3k+i,7k+i, 10k+i\, \vert \, 2 \leq i \leq k-3,\, i \text{ even}\}  \text{,}\\ 
 											x - 1, & $ for $ & x \in \{3k+i,7k+i, 10k+i\, \vert \, 3 \leq i \leq k-2,\, i \text{ odd}\}.
 	\end{array} \right.
 \]
 \begin{figure}[ht]
   \centering
 \includegraphics[scale=1]{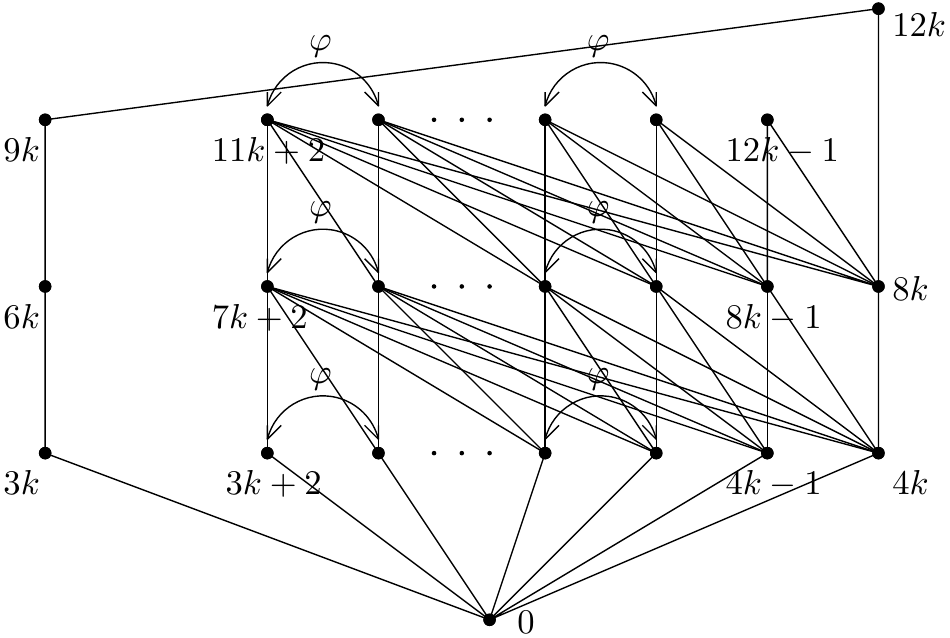}
 \caption{Involution of Proposition~\ref{juegaa+1aritm}.}\label{fig:phi2}
\end{figure}
 With Lemma~\ref{lemita} it is easy to see that $\varphi$ is an involution of ${\rm Ap}(\mS, 3k+1)$.
 This involution satisfies the hypotheses of Lemma~\ref{downsetstrat}. Hence, $3k+1$ is a winning move for $A$ if and only if 
 there is a winning strategy for the second player in the poset $(F,\lS)$, where $F$ is 
 the set of fixed points of $\varphi$. But, $F = \{0, 3k, 6k, 9k, 4k-1, 4k, 8k-1, 8k, 12k-1, 12k\}$
 and $(F, \lS)$ is isomorphic to the poset of Lemma~\ref{poset10elem}. Thus, $3k+1$ is a winning move for player $A$.
 \end{proof}

As stated in the beginning of the section we know the behavior of chomp on $\mS = \langle a, \ldots, 2a-1 \rangle$ and $\mS = \langle a, \ldots, 2a-2 \rangle$. The following result characterizes when $A$ has a winning strategy for chomp on $\mS = \langle a, \ldots, 2a-3 \rangle$.

\begin{prop}\label{tipo2consec}Let $\mS = \langle a, a+1,\ldots, 2a-3\rangle$ with $a \geq 4$. Player $A$ has a winning strategy for chomp on $\mS$ if and only if
$a$ is odd or $a = 6$.
\end{prop}
\begin{proof}If $a$ is odd, we are under the hypotheses of Proposition~\ref{juegaaaritm} and $a$ is a winning move for $A$.
If $a = 4$, then $\mS = \langle 4,5 \rangle$ is symmetric and $B$ has a winning strategy.  For $a = 6$, an exhaustive computer aided search shows that $36$ is a winning first
move for $A$, see Table~\ref{tab}.  

From now on we assume that $a$ is even and $a \geq 8$ and we are going to describe a winning strategy for $B$. First, we observe that for all $x \in \Z^+$ we have that $x \in \mS \Longleftrightarrow x \geq a$ and $x \notin \{2a-2,2a-1\}$. We partition $\mS$ into intervals $\mS_i := [a + i(2a+1),3a + i(2a+1)] \cap \mS$ for all $i \in \N$ and we denote by $(i,x)$ the element $a + i(2a+1) + x \in \mS_i$ for all $x \in \{0,\ldots,2a\}$ (see Figure~\ref{fig:si}).
\begin{figure}[ht]
  \centering
\includegraphics[scale=.7]{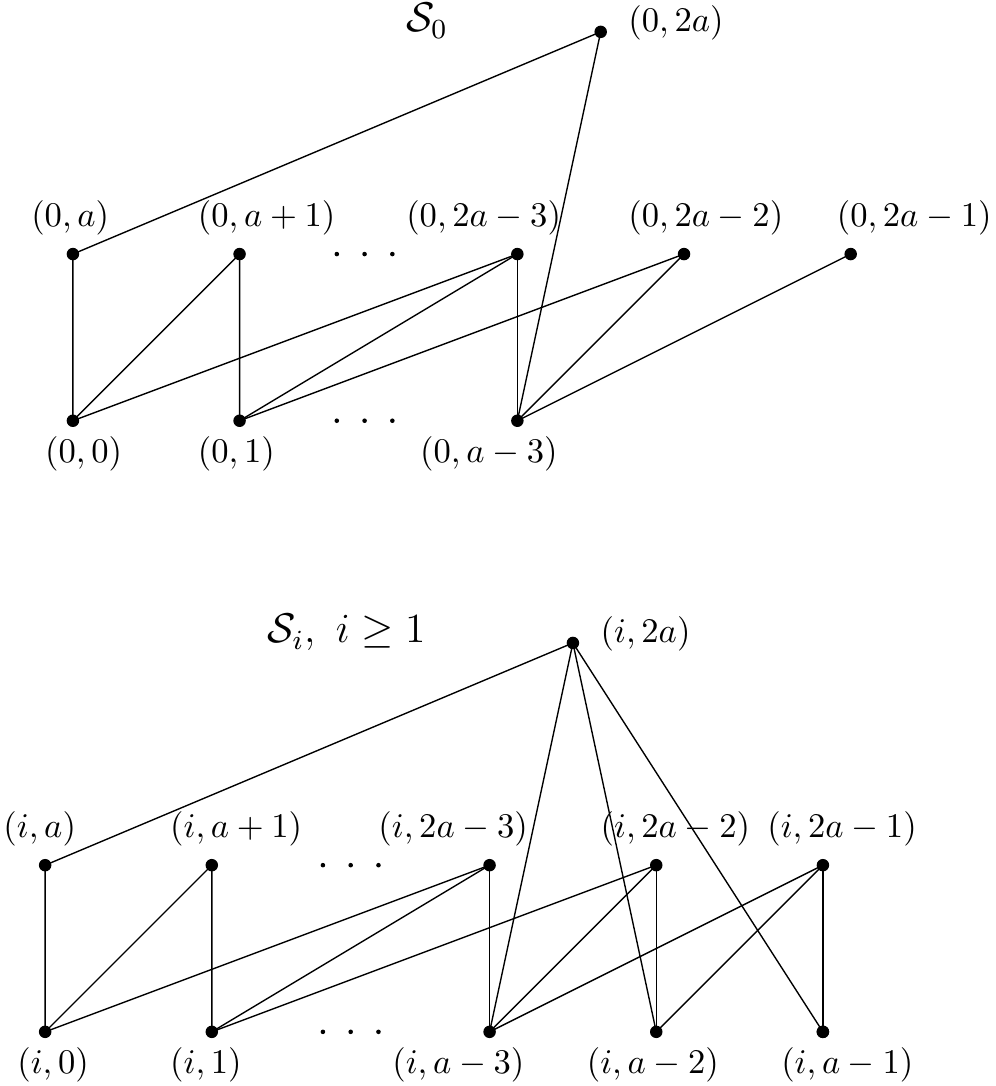}
\caption{$\mS_0$ and $\mS_i$ for $i \geq 1$.}\label{fig:si}
\end{figure}

We are going to exhibit a winning strategy for $B$ that respects the following rules.
\begin{itemize}
\item[(a)] $A$ is always forced to be the first to pick an element in $\mS_i$ for all $i$ (and will be finally forced to pick $0 \in \mS$).
\item[(b)] After $A$ picks the first element in $\mS_i$, $B$ answers with an element in $\mS_i$ and the remaining poset $P$ has at most one element in $\cup_{j > i} \mS_j$.
\item[(c)] $B$ avoids that after his move, the remaining poset $P$ satisfies that $(i,0), \ldots, (i,a-4) \notin P$ and $(i,a-3) \in P$ for every $i$.
\end{itemize}

Assume that $A$ picks an element for the first time in $\mS_i$ for some $i$. 
\begin{itemize}
\item[(1)] If $A$ picks $(i,0)$, then $B$ answers $(i,x)$ for any $x \in \{2,\ldots,a-2\}$; and vice versa. All these moves
leave the remaining poset 
\[ P = \{0\} \cup \bigcup_{j < i}\, \mS_j  \cup \{(i,y) \, \vert \, 1 \leq y \leq a-1, y \neq x\}. \] Hence, $P \cap \mS_i$ is just an antichain of $a-2$ elements; since $a-2$ is even, $B$ can easily satisfy (a). Moreover, in order to satisfy (c), in the next move player $B$ can pick $(i,a-3)$ if it had not been removed before.
\item[(2)] If $A$ picks $(i,a-1)$, then $B$ answers $(i,2)$. These moves leave the remaining poset 
\[ P = \{0\} \cup \bigcup_{j < i}\, \mS_j  \cup \{(i,y) \, \vert \, 0 \leq y \leq a+1, y \notin \{2,a-1\} \}. \]
Hence, $P \cap \mS_i$ is the poset of Figure~\ref{fig:caso2}. Since $a-4$ is even, $B$ can easily force $A$ to be the first to pick an element from $\mS_j$ for some $j < i$. Moreover, player $B$ can always assure condition (c) in his winning strategy; indeed, since $a \geq 8$, player $B$ can pick $(i,a-3)$ after the first time that $A$ picks $(i,y)$ for some $y \in \{3,\ldots,a-4,a-2\}$.
\begin{figure}[ht]
  \centering
\includegraphics[scale=1]{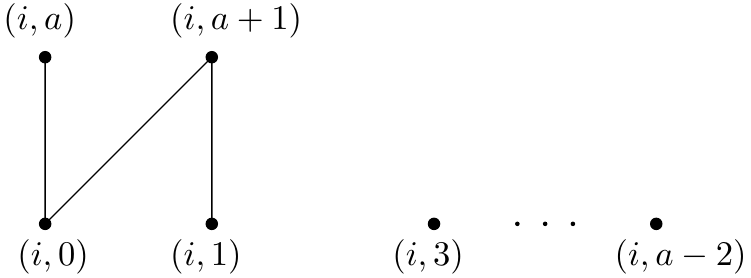}
\caption{$P \cap \mS_i$ in case (2).} \label{fig:caso2}
\end{figure}

\item[(3)] If $A$ picks $(i,1)$, then $B$ answers $(i,2a)$ and vice versa. These moves leave the remaining poset 
\[ P = \{0\} \cup \bigcup_{j < i}\, \mS_j  \cup \{(i,y) \, \vert \, y \in \{0,2,3,\ldots,a-1,a,2a-1\}\}. \]
Hence, $P \cap \mS_i$ is the poset of Figure~\ref{fig:caso3}. Again $B$ can easily satisfy (a) and (c).
\begin{figure}[ht]
  \centering
\includegraphics[scale=1]{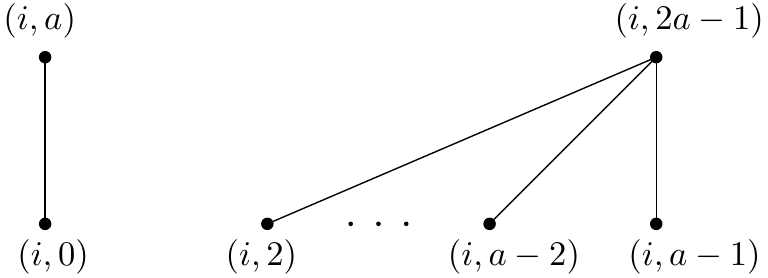}
\caption{$P \cap \mS_i$ in case (3).} \label{fig:caso3}
\end{figure}

\item[(4)] If $A$ picks $(i,a)$, then $B$ answers $(i,x)$ for any $x \in \{a+2,\ldots,2a-2\}$ and vice versa. These moves leave the remaining poset 
\[ P = \{0\} \cup \bigcup_{j < i}\, \mS_j  \cup (\mS_i \setminus \{(i,a),(i,2a),(i,x)\}. \]
Hence, $P \cap \mS_i$ is the poset of Figure~\ref{fig:caso4}. We are now providing a strategy for $B$ satisfying (a) and (c).
We divide the elements of $P \cap \mS_i$ into two sets, $L_1 := \{(i,y) \in P \cap \mS_i\, \vert \, y \leq a-1\}$ and $L_2 := \{(i,y) \in P \cap \mS_i\, \vert \, y \geq a\}$.
We observe that both $L_1$ and $L_2$ are of even size. If $A$ picks an element from $L_2$, then so does $B$. If $A$ picks an element from  $L_1$, then $B$ picks an element from $L_1$ such that after the two moves, every element of $L_2$ is removed (this can always be achieved by either picking $(i,1)$ or $(i,3)$). 

The second time that $A$ picks an element from $L_1$, $B$ picks $(i,a-3)$ if it has not been removed, yet. This assures (c). 
\begin{figure}[ht]
  \centering
\includegraphics[scale=.95]{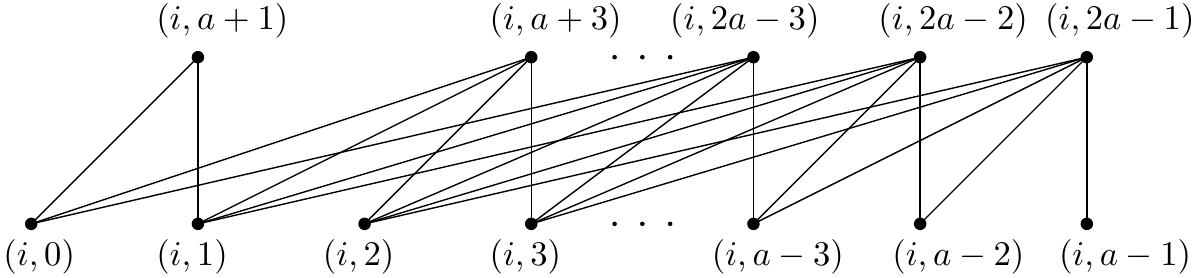}
\caption{$P \cap \mS_i$ in case (4) for $x = a+2$.} \label{fig:caso4}
\end{figure}

\item[(5)] If $A$ picks $(i,a+1)$, then $B$ answers $(i,a+3)$ and vice versa. These moves leave the remaining poset 
\[ P = \{0\} \cup \bigcup_{j < i}\, \mS_j  \cup (\mS_i \setminus \{(i,a+1),(i,a+3)\}. \]
Hence, $P \cap \mS_i$ is the poset of Figure~\ref{fig:caso5}. Let us see how $B$ can force (a) and (c).
If $A$ picks $(i,1)$ then $B$ answers $(i,2a)$ and vice versa; this leaves the poset in the same situation of case (3). If $A$ picks $(i,0)$; then $B$ can pick $(i,y)$ for any $y \in \{2,\ldots,a-1\}$ and vice versa; in this case the poset is left as in case (1). If $A$ picks $(i,a)$; then $B$ can pick $(i,y)$ for any $y \in \{a+2,a+4,\ldots,2a-1\}$ and vice versa; in this case the poset is left similar to  case (4) but with 2 elements less in $L_2$, the strategy for $B$ described in (4) also works here.  
\begin{figure}[ht]  \centering
\includegraphics[scale=1]{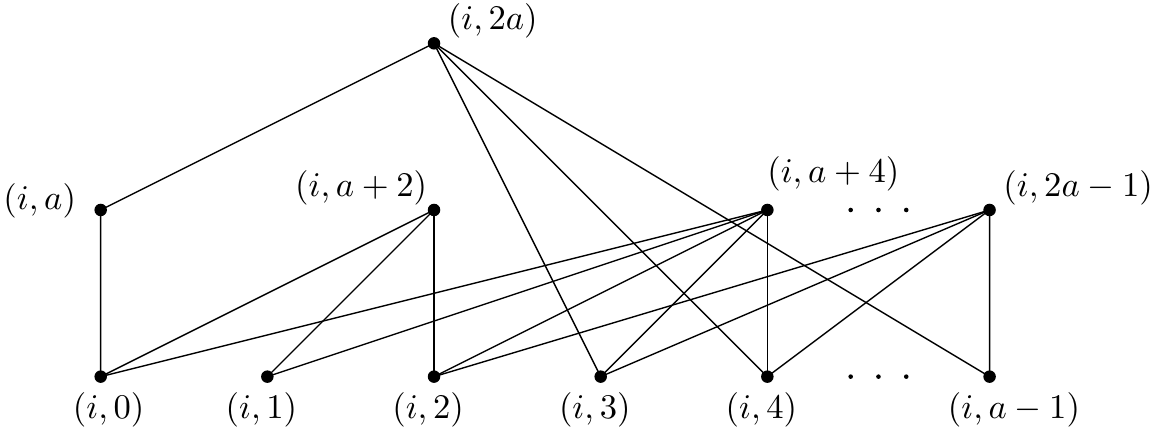}
\caption{$P \cap \mS_i$ in case (5).} \label{fig:caso5}
\end{figure}

\item[(6)] If $A$ picks $(i,2a-1)$, then we take $\lambda$ as the minimum value such that $(i+1,\lambda)$ is in the current poset. We separate two cases:
\begin{itemize} \item[(6.1)] If $\lambda$ does not 
exist or $\lambda \geq a-2$, then $B$ picks $(i,a)$ and the remaining poset is 
\[ P = \{0\} \cup \bigcup_{j < i}\, \mS_j  \cup (\mS_i \setminus \{(i,a),(i,2a-1),(i,2a)\}. \]
Hence, $P \cap \mS_i$ is the poset of Figure~\ref{fig:caso61}. In this case, the same strategy as in (4) works here to guarantee (a) and  (c).
\begin{figure}[ht]  \centering
\includegraphics[scale=1]{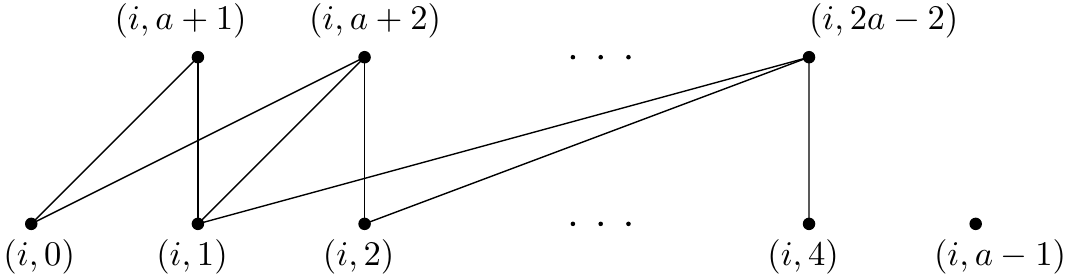}
\caption{$P \cap \mS_i$ in case (6.1).} \label{fig:caso61}
\end{figure}

\item[(6.2)] If $\lambda < a-2$, then $\lambda < a-3$ by condition (c) and $B$ can answer $(i,a+\lambda+2)$ because $a+\lambda+2 < 2a-1$; the remaining poset is 
\[ P = \{0\} \cup \bigcup_{j < i}\, \mS_j  \cup (\mS_i \setminus \{(i,a+\lambda+2),(i,2a-1)\} \cup \{(i+1,\lambda) \}. \]
Hence, $P' := P \cap (\mS_i \cup \mS_{i+1})$ is the poset of Figure~\ref{fig:caso62}. We partition $P'$ into three sets; $L_1 := \{(i,y) \in P'\, \vert \, 0 \leq y \leq a-1\}$,
$L_2 := \{(i,y) \in P'\, \vert \, a \leq y \leq 2a-2\}$ and $L_3 := \{(i,2a), (i+1,\lambda)\}$. Each of these three sets has an even number of elements and whenever $A$ picks an element from $L_i$, then $B$ can pick an element from $L_i$ such that $\cup_{j > i} L_j$ is completely  removed.  The only exception to this is if $A$ picks $(i,a-1)$, but in this case $B$ can pick $(i,2)$ and we are in case (2).
\begin{figure}[ht]  \centering
\includegraphics[scale=.9]{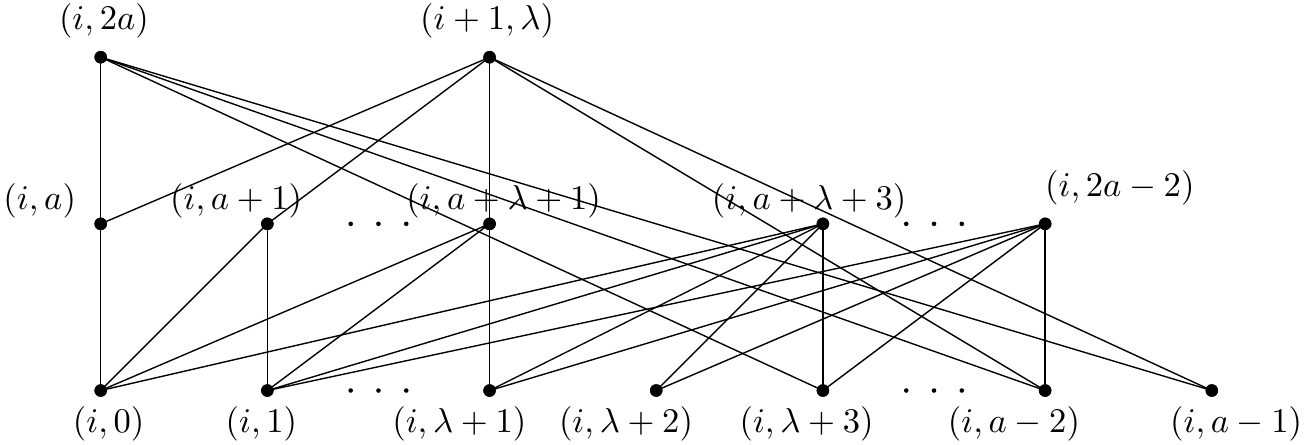}
\caption{$P \cap (\mS_i \cup \mS_{i+1})$ in case (6.2).} \label{fig:caso62}
\end{figure}

\end{itemize}
\end{itemize}
Thus, player $B$ has a winning strategy on $(\mS,\lS)$ when $\mS = \langle a,\ldots,2a-3\rangle$ with $a$ even and $a \geq 8$ and the proof is finished.
 \end{proof}

Condition (c) in the proof of Proposition~\ref{tipo2consec} might seem artificial but it is crucial for the proof. Indeed, if $(i,0), \ldots, (i,a-4) \notin P$, $(i,a-3) \in P$ and none of the players has picked an element from $\mS_{j}$ for all $j < i$; then $(i-1,2a-1) \in \mS_{i-1}$ is a winning move. In the proof, when we consider $a$ even and $a \geq 8$, the only place where
we use that $a \geq 8$ is to assure in case (2) that
the winning strategy can respect condition (c).  When $a = 6$ (and $\mS = \langle 6,7,8,9 \rangle$) this cannot be guaranteed.


\begin{rema}\label{table}
We illustrate in Table \ref{tab} which player has a winning strategy for interval generated semigroups with small multiplicity. 
The notation ${\mathbf B}$ means that player $B$ has a winning strategy because 
\begin{itemize} \item $\mS$ is symmetric (Theorem \ref{simetricoBgana}), 
\item $\mS$ is of maximal embedding dimension and $a$ even (Theorem \ref{medchomp}), or 
\item $k = a-3$ and $a \geq 8$ is even (Proposition \ref{tipo2consec}).
\end{itemize}

The notation ${\mathbf A}_{j}$ means that $A$ has a winning strategy whose first move is $j$, because 
\begin{itemize} 
\item[$\bullet$] $a$ odd and $k$ is even (Proposition \ref{juegaaaritm}); the subcases $a$ is odd and $k = a-1$ are also justified by Theorem \ref{medchomp},
\item[$\bullet$] $a = 3k$ and $k$ odd (Proposition \ref{juegaa+1aritm}), or
\item[$\bullet$] $k = a-3$ and $a$ is either odd or $a = 6$ (Proposition \ref{tipo2consec}). 
\end{itemize}

The notation ${\mathrm B}_{\leq i}$ means that an exhaustive computer search as mentioned in the beginning of the section shows that $B$ wins if $A$'s first move is $\leq i$. 
\end{rema}

A particular property of the semigroups considered in Proposition~\ref{tipo2consec} on which $B$ wins, is that they have type $2$ and $k$ is odd. Computational experiments suggest that when ${\rm type}(\mS) = 2$ and $k$ odd, then $B$ has a winning strategy in most cases, see Table~\ref{tab}. However, examples where this is not the case are given by $\mS = \langle 6,7,8,9 \rangle$ and $\langle 9,10,11,12 \rangle$, where in the latter $A$ wins playing $10$, by Proposition~\ref{juegaa+1aritm}. It would be interesting to completely characterize the cases where $A$ has a winning strategy in this setting.

\begin{table}
{\tiny 
\begin{tabular}{|c|c|c|c|c|c|c|c|c|c|c|c|c|c|c|c|}
\hline
$a\ \backslash\ k $ &  1 & 2 & 3 & 4 & 5 & 6 & 7 & 8 & 9 & 10 & 11 & 12 & 13 \\
\hline
2 &  $\mathbf{B}$ & & & & & & & & & & & & \\ \hline
3 &  $\mathbf{B} $ & $\mathbf{A}_{3} $  & & & & & & & & & & & \\ \hline
4 &  $\mathbf{B} $ & $\mathbf{B} $ & $\mathbf{B} $   & & & & & & & & & & \\ \hline
5 & $\mathbf{B} $ & $\mathbf{A}_{5} $ & $\mathbf{B} $ & $\mathbf{A}_{5} $   & & & & & & & & & \\ \hline
6 & $\mathbf{B} $ & $\mathbf{B} $ & $\mathbf{A}_{36} $ & $\mathbf{B} $ & $\mathbf{B} $   & & & & & & & & \\ \hline
7 & $\mathbf{B} $ & $\mathbf{A}_{7} $ & ${\mathrm B}_{\leq 49}$ & $\mathbf{A}_{7} $ & $\mathbf{B} $ & $\mathbf{A}_{7} $  & & & & & & & \\ \hline
8 & $\mathbf{B} $ & $\mathbf{B} $ & $\mathbf{B} $ & ${\mathrm B}_{\leq 43}$ & $\mathbf{B} $ & $\mathbf{B} $ & $\mathbf{B} $  & & & & & & \\ \hline
9 & $\mathbf{B} $ & $\mathbf{A}_{9} $ & $\mathbf{A}_{10}$ & $\mathbf{A}_{9} $ & ${\mathrm B}_{\leq 41}$ & $\mathbf{A}_{9} $ & $\mathbf{B} $ & $\mathbf{A}_{9} $   & & & & & \\ \hline
10 & $\mathbf{B} $ & $\mathbf{B} $ & ${\mathrm B}_{\leq 40}$ & $\mathbf{B} $ & ${\mathrm B}_{\leq 40}$ & ${\mathrm B}_{\leq 47}$ & $\mathbf{B} $ & $\mathbf{B} $ & $\mathbf{B} $ & & & & \\ \hline
11 & $\mathbf{B} $ & $\mathbf{A}_{11} $ & $\mathbf{B} $ & $\mathbf{A}_{11} $ & ${\mathrm B}_{\leq 42}$ & $\mathbf{A}_{11} $ & ${\mathrm B}_{\leq 43}$ & $\mathbf{A}_{11} $ & $\mathbf{B} $& $\mathbf{A}_{11} $   & & & \\ \hline
12 & $\mathbf{B} $ & $\mathbf{B} $ & ${\mathrm B}_{\leq 43}$ & ${\mathrm B}_{\leq 50}$ & $\mathbf{B} $ & ${\mathrm B}_{\leq 44}$ & ${\mathrm B}_{\leq 36}$ & ${\mathrm B}_{\leq 50}$ & $\mathbf{B} $& $\mathbf{B} $ & $\mathbf{B} $  & & \\ \hline
13 & $\mathbf{B} $ & $\mathbf{A}_{13} $ & ${\mathrm B}_{\leq 39}$ & $\mathbf{A}_{13} $ & ${\mathrm B}_{\leq 46}$ & $\mathbf{A}_{13} $ & ${\mathrm B}_{\leq 37}$ & $\mathbf{A}_{13} $ & ${\mathrm B}_{\leq 37}$ & $\mathbf{A}_{13} $ & $\mathbf{B} $ & $\mathbf{A}_{13} $ & \\ \hline
14 & $\mathbf{B} $ & $\mathbf{B} $ & $\mathbf{B} $ & $\mathbf{B} $ & ${\mathrm B}_{\leq 42}$ & $\mathbf{B} $ & ${\mathrm B}_{\leq 42}$ & ${\mathrm B}_{\leq 49}$ & ${\mathrm B}_{\leq 40}$& ${\mathrm B}_{\leq 50}$ & $\mathbf{B} $ & $\mathbf{B} $ &$\mathbf{B} $  \\ \hline
\end{tabular}}
\vspace{.2cm}
\caption{Winner of chomp on $\mS = \langle a, \ldots, a + k \rangle$ for small values of $a$ (see Remark \ref{table}).}
\label{tab}
\end{table}

Another observation coming from our results on $\mS = \langle a, \ldots, 2a-c \rangle$ for $1\leq c\leq 3$ is that for interval generated semigroups for which $a+k$ is close to $2a$, player $A$ seems to win only in the case guaranteed by Proposition~\ref{juegaaaritm}, that is, $a$ odd and $k$ even. We formulate this suspicion as Conjecture~\ref{conjecture} in the last section.


\section{A bound for the first move of a winning strategy for $A$}\label{sec:bound}

The goal of this section is to prove that for any numerical semigroup $\mS$, there is a value $\Delta$ depending only on the Frobenius number and the number of gaps of $\mS$ such that if $A$ has a winning strategy, then $A$ has
a winning strategy with a first move $\leq \Delta$. The dependence is exponential in the Frobenius number of the semigroup
$g(\mS)$ and doubly exponential in the number of gaps of the semigroup, which we denote by $n(\mS)$. We recall that the {\it set of gaps} of a numerical semigroup is the
finite set $\mathcal N(\mS) = \N \setminus \mS = \{x \in \N \, \vert \, x \notin \mS\}$.

The proof exploits the ``local" behavior of the poset, in the sense that every down-set $P' \subsetneq \mS$ can be encoded by a pair $(x,C)$, where $x \in \mS$ and $C \subset \mathcal N(S)$. Indeed, all proper down-sets  $P' \subsetneq \mS$ are finite and if we take $x$ as the minimum integer in $\mS \setminus P'$, then we claim that $P'$ satisfies that $$([0, x-1] \cap \mS) \subseteq P' \subseteq ([0, x-1] \cup (x + \mathcal N(S))) \cap \mS,$$
where $x + \mathcal N(S) := \{x + n \, \vert \, n \in \mathcal N(\mS)\}$. Since the inclusions $([0, x-1] \cap \mS) \subseteq P' \subseteq \mS$ are trivial, let us prove 
that $P' \subseteq ([0, x-1] \cup (x + \mathcal N(S))).$ For this purpose, we take $y \in P'$ such that $y \geq x$. Since $P'$ is a down-set and $x \notin P'$, we have $x \not\lS y$. Hence, $y - x \in \mathcal N(S)$ and the claim is proved.

Thus, one can choose $C \subseteq \mathcal N(S)$ in such a way  that $P' = ([0, x-1] \cap \mS) \cup (x + C)$. We will denote this by $P' \equiv (x,C)$. It is evident that
the down-sets of $(\mS, \lS)$ are in bijection with the possible states of the chomp game, hence we encode the state of a game by elements of $\mS \times \frak P(\mathcal N(S))$, where $\frak P(\mathcal N(S))$ denotes the power set of $\mathcal N(S)$.

Let us illustrate this way of encoding the states of the chomp game with an example.

\begin{exam}\label{ex:1} 
Let $\mS = \langle 3, 5 \rangle$, then the set of gaps of this semigroup is $\mathcal N(\mS) = \{1,2,4,7\}$.

Assume now that the first player
picks $8 \in \mS$, then the remaining poset is $P_1 := \mS \setminus (8 + \mS) = {\rm Ap}(\mS,8) = \{0,3,5,6,9,10,15\}$. We observe that $[0, 7] \cup (7 + \mathcal N(S)) = [0,7] \cup \{9,10,12,15\}$ and, thus, $P_1 =  ([0, 7] \cup (8 + \mathcal N(S))) \cap \mS \equiv (8,\mathcal N(\mS)).$

If now player $B$ picks $9$, then the remaining poset is $P_2 := P_1 \setminus (9 + \mathcal \mS) = \{0,3,5,6,10\}$. The smallest element already picked is $x = 8$
and $P_2 = ([0, 7] \cup (8 + \mathcal \{2\})) \cap \mS \equiv (8, \{2\}).$

\end{exam}

Assume that at some point of the game, the remaining poset is $P \equiv (x,C)$ with $x \in \mS$ and $C \subseteq \mathcal N(S)$, and the player to play picks $x' \in P$. The following lemma explains how the resulting poset $P' := P \setminus (x' + \mS)$ is encoded.

\begin{lemma} \label{lemacodif} Let $(\mS,\lS)$ be a numerical semigroup poset, $P \equiv (x,C) \subsetneq \mS$ a down-set, and $x' \in P$.  Then, $P \setminus (x' + \mS)$ is 
encoded by $(x'',D)$, where
\begin{itemize} \item[(a)] $x'' = \min\{x,x'\}$,
\item[(b)] if $x < x'$, then $D \subsetneq C$,
\item[(c)] if $g(\mS) < x' < x - g(\mS)$, then $(z,D) = (x',\mathcal N(S))$. 
\end{itemize}
\end{lemma}
\begin{proof} By definition $x''$ is the smallest integer in $\mS \setminus P$, and equals $\min\{x,x'\}$. If $x < x'$, then $P' = (x,C) \setminus (x' + \mS) = (x, D)$ with 
$D \subseteq C$. Since $x' \in C \setminus D$, then $D \subsetneq C$. To prove (c) we observe that $x' < x$, then by (a) we have that $P = (x', D)$ for some $d \subset \mathcal N(\mS)$. To prove that $D = \mathcal N(\mS)$ it suffices to verify that $x' + z \in P'$ for all $z \in \mathcal N(\mS)$. Indeed, since $g(\mS) < x' + z \leq x' + g(\mS) < x''$, then $x' + z \in P$. Moreover $x' + z \notin x' + \mS$, so $x' + z \in P \setminus (x' + \mS) = P'$ and (c) is proved. 
 \end{proof}

The following technical lemma, whose proof is straightforward, will be useful in the proof of the main result of this section.

\begin{lemma}\label{pasalomismo}  Let $(\mS,\lS)$ be a numerical semigroup poset and take $C \subseteq \mathcal N(\mS)$ and $x, x', y, y' > g(\mS)$ such that $y' -y = x' - x$. We set $P \equiv (x,C)$ and $Q \equiv (y,C)$. If $x' \in P$ and $P \setminus (x' + \mS) \equiv (\min\{x,x'\},D)$, then: 
\begin{itemize} 
\item[$\bullet$] $y' \in Q$,
\item[$\bullet$] $Q \setminus (y' + \mS)  \equiv (\min\{y,y'\}, D)$. 
\end{itemize}
\end{lemma} 

Let us illustrate this result with an example.

\begin{exam}
Consider again $\mS = \langle 3, 5 \rangle$. We saw that if we take $x = 8,\, x' = 9$ and $C = \mathcal N(\mS) = \{1,2,4,7\}$, then $(x,C) \setminus (x' + \mS) = (x, \{2\})$.
Setting $y = 8 + \lambda,\, y' = 9 + \lambda$ for some $\lambda \in \Z^+$, it is easy to see that $(y,C) \setminus (y' + \mS) = (y, \{2\})$.
\end{exam}

Now we can proceed with the main result of this section.

\begin{theorem}\label{boundonstrategy}
If $A$ has a winning strategy for chomp on $\mS$, then $A$ has a winning strategy with first move less than $2^{g 2^n }$, where $g = g(\mS)$ and $n = n(\mS)$.
\end{theorem}
\begin{proof} Denote $\Delta := 2^{g 2^n}$ and suppose that $A$ has no winning strategy with first move less than $\Delta$. We will prove that $A$ has no winning strategy.

For every $x \in \mS$ we define the set $W_x \subseteq \frak P (\mathcal N(\mS))$ as follows: $C \in W_x$ if and only if the first player has a winning strategy
on the chomp game on the poset $(x,C)$. The elements $C \in W_x$ can be inductively characterized as follows: $C \in W_x$ if and only if there exists $y \in (x,C)$ such that
if $(x,C) \setminus (y + \mS) = (z,D)$, then $D \notin W_z$. Another interesting property of the sets $W_x$ is that, if $x > g$, then $x$ is a winning first move for chomp on $\mS$ if and only if $\mathcal N(\mS) \notin W_x$.  In particular, since $A$ has no winning strategy with first move less than $\Delta$, we have that $\mathcal N(\mS) \in W_x$ for all $g < x \leq \Delta$.

The result is a direct consequence of the following claims.

{\it Claim 1}: Let $x,y \in \mS:\, g < x < y \leq \Delta$ and $k \in \Z^+$ such that $W_{x+i} = W_{y+i}$ for all $0 \leq i < k$. Then, $z \in \mS$ is not a winning first move for player $A$ for all  
$z < y +k$.

\noindent {\it Proof of Claim~1}. Assume there exists a winning first move $z < y + k$, and take $z$ the minimum possible.
We observe that $g < x < y \leq \Delta \leq z < y + k$ and that $\mathcal N(\mS) \notin W_z = W_{y + (z - y)}$ with $z - y < k$. Hence, $W_{y + (z-y)} = W_{x+(z-y)}$ and $\mathcal N(\mS) \notin W_{x + (z - y)}$. This implies that $x + z - y$ is a winning first move but $x + z - y < z$, a contradiction.

{\it Claim 2:} There exist $x,y \in \mS:\, g < x < y \leq \Delta$ such that  $W_{x+i} = W_{y+i}$ for all $0 \leq i < g$

\noindent {\it Proof of Claim~2}.
Since the power set of $\mathcal N(\mS)$ has $2^n$ elements and $\mathcal N(\mS) \in W_{x}$ for all $x:\, g < x \leq \Delta$;  
there are $2^{2^n-1}$ possible
sets $W_x$ for each $x: \, g < x \leq \Delta$. As a consequence, there are  $\mu := 2^{(2^{n}-1)g}$ possible combinations of sets $(W_x, W_{x+1},\ldots,W_{x+g-1})$ for all $x: 
\, g < x \leq \Delta - g$. Since $\mu < \Delta - 2g$, there exist
$x, y:\, g < x < y \leq \Delta - g$ such that $(W_{x},W_{x+1},\ldots,W_{x+g-1})  = (W_{y},W_{y+1},\ldots,W_{y+g-1})$.

{\it Claim 3}: Take $x,y$ as in {\it Claim 2}, then $W_{x+i} = W_{y+i}$ for all $i \in \N$.
 
\noindent {\it Proof of Claim~3}. By contradiction, we take $k$ as the minimum value such that $W_{x+k} \neq W_{y+k}$. By {\it Claim 2} we have that $k \geq g$.

Now, take $C$ as an element in the symmetric difference of $W_{x+k}$ and $W_{y+k}$ with the minimum number of elements. Assume without loss of generality that $C \in W_{x+k}$ (and $C \notin W_{y+k}$). 

Since $C \in W_{x+k}$, there is a winning move in the poset $(x + k,C)$. Let $x'$ be this winning move. We are going to prove that $y' := x' - x + y$ is a winning move for $(y + k, C)$; this would be a contradiction because $C \notin W_{y+k}$ and, hence, there is no winning move in this poset. 
We denote $(x + k, C) \setminus (x' + \mS) \equiv (x'',D)$; we have that 
\begin{itemize} \item[(i)] $x'' = \min\{x,x'\}$, by Lemma \ref{lemacodif}.(a), 
\item[(ii)]  $D \notin W_{x''}$, because $x'$ is a winning answer in $(x+k,C)$, and
\item[(iii)] $x' \geq x + k - g$ (otherwise, by Lemma \ref{lemacodif}.(c), $x'' = x'$ and $D = \mathcal N(\mS) \notin W_{x'}$, which means that $x'$ is a winning first move; a contradiction to {\it Claim 1}). 
\end{itemize}

We set $y'' := x'' - x + y$, by Lemma \ref{pasalomismo} we have that $(y + k, C) \setminus (y' + \mS) = (y'', D)$. To prove that $y' = x' - x + y$ is a winning move for $(y + k, C)$ is equivalent to see that $D \notin W_{y''}$. We separate two cases:
\begin{enumerate}
 \item $x' > x + k$: We have that  $x'' = x + k$ and $D \subsetneq C$ by Lemma \ref{lemacodif}.(b); hence $(y'', D) = (y+k, D)$. Since $D \subsetneq C$ and $D \notin W_{x+k}$, we get that $D \notin W_{y+k} = W_{y''}$.
 \item $x' < x + k$: We have that $x'' = x'$. By (iii) we have that $x \leq x + k - g \leq x' < x + k$, then  
$W_x' = W_{x + (x'-x)} = W_{y+(x'-x)} = W_{y+(y'-y)} = W_{y'}$ and $D \notin W_{y'} = W_{y''}$. 
\end{enumerate}

This proves Claim~3. Claims~1 and 3 conclude the proof.
 \end{proof}

Apart from this bound being probably far from optimal, its theoretic use is to prove the decidability of the problem of determining which player has a winning strategy. Indeed, 
this result yields a (non-efficient) algorithm for this problem: it suffices to try all possible games with a first move $\leq \Delta$ to decide if player 
$A$ has a winning strategy. It would be nice to have better bounds on a first move for $A$ which uses in a deeper way the semigroup structure of the poset. 

Indeed, it is worth pointing out that the proof only exploits the local behavior of the poset and, hence, such a result can
 be obtained for infinite posets with similar behavior. This is the purpose of the following section.

\section{Numerical semigroups with torsion}\label{sec:torsion}

The goal of this section is to extend Theorem~\ref{simetricoBgana} and Theorem~\ref{boundonstrategy} to more general families of posets with a semigroup structure. We start from a high point of view and consider subsemigroups of $\N \times T$ where $T$ is any finite monoid. The main result of this section is that the above-mentioned theorems generalize naturally to this setting if and only if $T$ is a finite abelian group.  Semigroups of this kind, which we call {\it numerical semigroups with torsion}, have been considered in the literature by several authors due to their connections with one dimensional lattice ideals (see, for example, \cite{HV,MT}). Whenever $\mS$ is a numerical semigroup with torsion that additionally has no inverses, it carries a natural poset structure $(\mS,\lS)$ induced by $x \lS y \Leftrightarrow y - x \in \mS$. Moreover, for every $x, y \in \mS$ the upper sets $x + \mS$ and $y + \mS$ are isomorphic with via the addition of $y - x \in \Z \times T$. Since the addition in $\Z \times T$ is commutative this implies that $(\mS, \lS)$ is auto-equivalent in the sense of \cite[Section 5]{cgpr}.

In order to extend Theorem~\ref{simetricoBgana} and Theorem~\ref{boundonstrategy} let us first highlight the main ingredients of their proofs. 
We observe that Theorem~\ref{simetricoBgana} relies on the following property of semigroup posets: whatever player $A$'s first move is, the remaining poset is finite (Remark~\ref{aperychomp}) and its number of maximal elements does not depend on the first move (Corollary~\ref{typeinvariant}). Moreover, this number is exactly $1$ if and only if the corresponding numerical semigroup is symmetric. Similarly, Theorem~\ref{boundonstrategy} builds on the local behavior of numerical semigroup posets. Hence,  whenever a family of posets has this behavior, we obtain analogue versions of Theorems~\ref{simetricoBgana} and~\ref{boundonstrategy}. 


In order to see how far we can go with respect to the structural restrictions on $T$ while still ensuring a generalization of our results, we start by studying submonoids of $(\N \times T, \cdot)$, where $\N$ is considered as an additive monoid and $(T,\cdot)$ is a finite monoid. In order to avoid  cases, we assume that our numerical semigroups with torsion are infinite, that is, not all their elements have a zero in the first coordinate.

Even if we will not make explicit use of it later, we begin by studying which such monoids are finitely generated as a result of independent interest. For this we restrict the notion of finitely generated in the following sense: Let $T$ be a finite monoid with neutral element $e$ and let $(\mS,\cdot)$ be a submonoid of $\N \times T$. We set $\mS_{t} := \mS \cap (\N \times \{t\})$ for all $t \in T$. We say that $\mS$ is \emph{nicely generated} if there exists a finite set $F \subset \mS$ such that $ \mS = F \cdot \mS_e$. Since $\mS_e$ is 
isomorphic to a subsemigroup of $\N$, there exist $a_1,\ldots, a_n \in \N$ such that $\mS_e = \langle (a_1,e),\ldots, (a_n,e) \rangle$. Thus, $\mS$ being nicely generated means that every $y \in \mS$ can be written as $y = x \cdot (b,e)$ with $x \in F$ and $b \in \langle a_1,\ldots,a_n \rangle$. In particular, nicely generated implies finitely generated.  

\begin{exam} \label{ex:2}
 Consider $(T,\cdot)$ the monoid
with $T = \{0,1,2\}$ and $i \cdot j := \min \{i+j,2\}$ for all $i,j \in \{0,1,2\}$. Then $\mS = \langle (3,0), (2,1), (3,2) \rangle$ is nicely generated, indeed, $\mS_e = \langle (3,0) \rangle$ and $\mS = F \cdot \mS_e$ with $F = \{(0,0), (2,1), (3,2), (4,2), (8,2)\}$
\end{exam} 

\begin{prop}\label{conditionfingen}Let $T$ be a finite monoid with neutral element $e$ and let $(\mS,\cdot)$ be a submonoid of $\N \times T$. We have that $\mS$ is nicely generated if and only if $(a,e) \in \mS$ for some $a > 0$. 
\end{prop}
\begin{proof} 
The first direction is trivial. If $\mS$ does not  contain an element of the form $(a,e)$ with $a > 0$, then $\mS = F \cdot \mS_e = F$ is finite, contradicting our general assumption that $\mS$ is infinite.

Conversely, we take $a_1,\ldots,a_n \in \N$ such that $\mS_e = \langle (a_1,e), \ldots, (a_n,e) \rangle$. If we set $d := \gcd\{a_1,\ldots,a_n\}$, then
 there exists $g \in \Z^+$ such that for all $b > g$, $(b,e) \in \mS_e$ if and only if $d$ divides $b$.

For every $t \in T$, $j \in \{0,\ldots,d-1\}$ we denote $\mS_{t,j} := \{(x,t) \in \mS_t\, \vert \, x \equiv j \ ({\rm mod}\ d)\}$ and, whenever $\mS_{t,j}$ is not empty, we set $r_{t,j} := \min \{x \in \N \, \vert \, (x,t) \in \mS_{t,j} \}.$ We conclude that there is only a finite number of elements in $\mS$ that cannot be expressed as $(r_{t,j},t) \cdot (s,e)$ with $(s,e) \in \mS_e$; this proves the result.
 \end{proof}

 
 As a consequence of Proposition \ref{conditionfingen} we get  that if $T$ is a finite group, then $\mS$ is always finitely generated.
 
\begin{coro}\label{corfingen} Let $T$ be a finite group with neutral element $e$ and let $(\mS,\cdot)$ be a submonoid of $\N \times T$. Then $\mS$ is nicely generated.
 \end{coro}
 
\begin{proof}By Proposition \ref{conditionfingen}, if there exists $a > 0$ such that $(a,e) \in \mS$, then  we are done. Otherwise, we are going to see that $\mS \subset \{0\} \times T$ and, hence, it is finite. By contradiction, assume that there exists $(b,t) \in \mS$ with $b > 0$;  if we multiply $(b,t)$ with itself $k$ times, with 
$k$ the order of $t \in T$, we get that $(k b, e) \in \mS$ and $kb > 0$, a contradiction.
  \end{proof}

From now on, we assume without loss of generality that $\mS_t \not= \emptyset$ for all $t \in T$; otherwise one could find another monoid $T' \subsetneq T$ such that
$\mS \subseteq \N \times T'$. For example for $T = \Z / 6 \Z$ the cyclic group with $6$ elements, the semigroup $\mS = \langle (7,2), (4,4) \rangle \subset \N \times T$ satisfies  that $\mS_{1} = \mS_{3} = \mS_{5} = \emptyset$; hence taking $T' := 2 \Z / 6 \Z = \{0,2,4\}$ we have that $\mS \subset \N \times T'$.

In general, the complement of $x \cdot \mS := \{x \cdot s \, \vert \, s \in \mS\}$ in $\mS$ is not finite. For instance, if we take $x = (2,1)$ in Example \ref{ex:2}, then
$\langle (3,0) \rangle$ is an infinite subset of $\mS \setminus (x \cdot \mS)$. 
Proposition~\ref{finiteafteronemove} provides a sufficient condition
so that the complement of $x \cdot \mS := \{x \cdot s \, \vert \, s \in \mS\}$ in $\mS$ is finite for all $x \in \mS$. Since the set $\mS \setminus (x \cdot \mS)$ is a natural
generalization of the Ap\'ery set for numerical semigroups, we will usually denote it by ${\rm Ap}(\mS,x)$.
The condition that ${\rm Ap}(\mS,x)$ is finite for every $x$ is equivalent to saying that when playing chomp
on $\mS$, after one move the remaining poset is finite. Before stating the result we recall a well known result.

\begin{prop}\label{latinsquare}Let $T$ be a monoid. Then, $T$ is a group if and only if for all $t_1,t_2 \in T$ there exists an $u \in T$ such that $t_1 \cdot u = t_2$.
\end{prop}

%
\begin{prop}\label{finiteafteronemove}Let $(T,\cdot)$ be a finite monoid and let $\mS$ be a submonoid of $\N \times T$. We have that $T$ is a group if and only if ${\rm Ap}(\mS,x)$ is finite for all $x \in \mS$ and $\mS$ is nicely generated.
\end{prop}

\begin{proof}For the first direction let $T$ be a group. By Corollary~\ref{corfingen} we have that $\mS$ is nicely generated. Hence, there exists an element $(a,e) \in \mS$ with $a > 0$. We write $\mS_e = \langle (a_1,e),\ldots,(a_n,e) \rangle$ for some $a_1,\ldots,a_n \in \Z^+$. If we set $d := \gcd\{a_1,\ldots,a_n\}$, then:
\begin{enumerate} 
\item[(1)] if $(b,e) \in \mS_e$, then $b$ is a multiple of $d$, and
\item[(2)] there exists $g \in \Z^+$ such that for all $b > g$, $(b,e) \in \mS_e$ if and only if $b$ is a multiple of $d$.
\end{enumerate}
For every $(b_1,t), (b_2,t) \in \mS$ we claim that:
\begin{itemize}
\item[(3)] $b_1 - b_2$ is a multiple of $d$;
\end{itemize} indeed, there exists a $c \in \N$ such that $(c,t^{-1}) \in T$, then $(b_1+c,e), (b_2+c,e) \in \mS_e$ and (3) follows from (1).

Take now $x = (b,t) \in \mS$. We observe that ${\rm Ap}(\mS,x) = \cup_{s \in T} \left( \mS_s \setminus (x \cdot \mS) \right)$ and we
are going to prove that $\mS_s \setminus (x \cdot \mS)$ is finite for all $s \in T$. More precisely, if we take $x' = (c,t') \in \mS$ so that $t \cdot t' = s$,
we are proving that $(a,s) \in x \cdot \mS$ for all  $a > b + c + g$. Indeed, since $x \cdot x' = (b+c,s)$, then $a - b - c$ is a multiple of $d$ by (3); and, since $a - b - c > g$, then $(a-b-c,e) \in \mS_e$ by (2). Thus, we can write $(a,s) = x \cdot x' \cdot (a-b-c, e) \in x \cdot x' \cdot \mS_e \subset x \cdot \mS$, and the result follows.

We prove the second direction by contradiction. Assume that $T$ is not a group. Since $\mS$ is nicely generated by Proposition~\ref{conditionfingen} there exists $a > 0$ such that $(a,e) \in \mS$.  First, we observe that $\mS_t$ is infinite for all $t \in T$. Indeed, for all $t \in T$, if we take $b \in \N$ such that $(b,t) \in \mS_t$, then $(b + \lambda a, t) \in \mS_t$ for all $\lambda \in \N$.
Since $T$ is not a group, by Proposition~\ref{latinsquare}, there exist $t_1,t_2$ such that $t_1 \cdot u \neq t_2$ for all $u \in T$.  Then, for all
$x \in \mS_{t_1}$ we have that $\mS_{t_2} \subset {\rm Ap}(\mS,x)$ and, thus, it is infinite.
 \end{proof}

The following example illustrates why the hypothesis of being nicely generated cannot be removed. Consider $(T,\cdot)$ the monoid
with $T = \{0,\ldots,n\}$ and $i \cdot j := \min \{i+j,n\}$ for all $i,j \in \{0,\ldots,n\}$. Clearly, $T$ is not a group but the monoid $\mS \subset \N \times T$ with elements $\{(i,i) \,\vert \, 0 \leq i < n\} \cup \{ (i,n) \, \vert \, i \in \N\}$ satisfies that for all $x \in \mS$, the set ${\rm Ap}(x,\mS)$ is finite. While $\mS$ is finitely generated, it does not contain any $(a,0)$ with $a>0$ and therefore is not nicely generated by Proposition~\ref{conditionfingen}.

Motivated by Proposition~\ref{finiteafteronemove}, from now on we will consider submonoids $\mS$ of $\N \times T$, where $T$ is a finite group. It is worth pointing out that such a monoid is embeddable in the group $\Z \times T$.  We are first proving that in this context one can extend the definition of Frobenius number.

\begin{prop}\label{frobgeneral}Let $(T,\cdot)$ be a finite group and $\mS$ a submonoid of $\N \times T$. Denote by $\Z \mS$ the smallest subgroup of $\Z \times T$ containing $\mS$.
There exists $g(\mS) \in \Z$ such that for all $(a,t) \in \Z \mS$ if $a > g(\mS)$, then $(a,t) \in \mS$. 
\end{prop}
\begin{proof}We take $G$ the smallest subgroup of $\Z \times \{e\}$ containing $\mS_e$. Since $G \subset \Z \times \{e\}$, it is cyclic with generator $(d, e)$ for some $d \in \N$. If $d = 0$, then $\mS = G = \{0\} \times T$ and the result follows. Assume that $d > 0$, then there exists $g_e \in \N$ such that for all $(a,e) \in G$ with $a > g_e$, we have that $(a,e) \in \mS$. For all $t \in T$ we take $m_t := \min\{a \in \N \, \vert \, (a,t) \in \mS_t\} $. 
We claim that $g(\mS) := g_e  + \max\{m_t\, \vert \, t \in T\}$ satisfies the desired property. Consider $(c,t) \in \Z \mS$ with $c > g(\mS)$. We have that
$(m_t, t) \in \mS$ and then $(c,t) \cdot (m_t,t)^{-1} = (c-m_t, e) \in \Z \mS$. Moreover, since $c-m_t > g_e$, we have that $(c-m_t,e) \in \mS$. Thus, $(c,t) = (m_t,t) \cdot (c-m_t, e) \in \mS$. 
 \end{proof}

Let us illustrate this result with an example.

\begin{exam}
 Let $T = S_{3}$ be the symmetric group on $3$ elements and consider the monoid $\mS = \langle (3, id), (2, (12)), (4,(123)) \rangle \subset \N \times S_3$. Since the set
$\{(12), (123)\}$ generates $S_3$, we have $\mS_t \neq \emptyset$ for all $t \in S_3$. Moreover, the equality 
$(2, (12)) \cdot (2,(12)) \cdot (3, id)^{-1} = (1,id)\in \Z \mS$ yields that $\Z\mS = \Z \times S_3$. Since
$(3,id), (4,id) \in \mS$ there exists $g(\mS) \in \ \Z^+$ such that for all $(a, \sigma) \in \Z \times S_3$ with $a > g(\mS)$, we have $a \in \mS$. More precisely, we have that $(2,(12))$, $(4,(123))$, $(6,(23))$, $(6,(13))$, $(8,(132)) \in \mS$ and that $\mS_{id} = \langle (3,id), (4,id) \rangle$, thus Proposition \ref{frobgeneral} holds with $g(\mS) = 13$.
\end{exam}

We define the {\it set of gaps} of the semigroup $\mS$ as $\mathcal N(\mS) := \Z \mS \cap (\N \times T)\setminus \mS$. Proposition~\ref{frobgeneral} implies that this set is always finite.


To play chomp on $\mS$, we need that $(\mS,\lS)$ has a poset structure with a global minimum. The role of the minimum will be played by $(0,e) \in \mS$, the neutral element. We consider the binary relation $r_1 \leq_\mS r_2 \Leftrightarrow r_2 \cdot r_1^{-1} \in \mS$. A necessary and sufficient condition for $\leq_\mS$ to be an order is that \begin{equation}\label{eq:antisym} (0,t) \in \mS \Longleftrightarrow t = e. \end{equation} Indeed, $\leq_\mS$ is reflexive (because $\mS$ has a neutral element), transitive (because $\mS$ is closed under $\cdot$) and property (\ref{eq:antisym}) implies  antisymmetry.  A monoid satisfying property (\ref{eq:antisym}) will be called an {\it ordered monoid}.

In order to extend Theorem~\ref{simetricoBgana} we also need that the number of maximal elements of the remaining poset is invariant under the choice of the first move. We will see that this is always the case if and only if $T$ is commutative, that is, if $\mS$ is a numerical semigroup with torsion.

\begin{prop}\label{maxinvariantelattice}Let $T$ be a finite group.
The number of maximal elements of ${\rm Ap}(\mS,x)$ is invariant under the choice of $x$ for every ordered monoid $(\mS,\leq_{\mS})$ with $\mS \subset \N \times T$ if and only if $T$ is commutative.
\end{prop}
\begin{proof}
$(\Leftarrow)$ We assume that $T$ is commutative. To prove the result it suffices to establish a bijection between the set of maximal gaps of $\mS$ with respect to $\leq_\mS$ and the set of maximal elements of ${\rm Ap}(\mS,x)$. In fact, it is straightforward to check that the map sending $\sigma: \mathcal N(\mS) \longrightarrow {\rm Ap}(\mS,x)$ defined as $y \mapsto x \cdot y$ is a bijection.

$(\Rightarrow)$ Suppose that $T$ is not commutative. We are going to exhibit an ordered monoid $(\mS, \leq_{\mS})$ with $\mS \subset \N \times T$ and two elements $x,y \in \mS$ such that the number of maximal elements of ${\rm Ap}(\mS,x)$ and  ${\rm Ap}(\mS,y)$ with respect to $\leq_{\mS}$ is not the same. We take $s,t \in T$ such that $s \cdot t \neq t \cdot s$ and we set 
\[ \mS := \{(i,s) \, \vert \, i \geq 1\} \cup \{(i,s \cdot t) \, \vert \, i \geq 3\} \cup \left( \cup_{u \notin \{s,s \cdot t\}} \{(i,u) \, \vert \, i \geq 2\}\right) \cup \{(0,e)\} \]   

We observe that $\mS$ is a semigroup, the only potential problem would be if $(1,s) \cdot (1,s) = (2, s \cdot t)\notin \mS$, but  this is not possible because $s^2 \neq s\cdot t$.
Now we choose $x := (2,e)$ and $y := (1,s)$. We are going to prove that ${\rm Ap}(\mS,x)$ has more maximal elements than ${\rm Ap}(\mS,y)$. Indeed, this is a consequence of the following facts:
\begin{itemize}
\item [(a)] ${\rm Ap}(\mS,x) =  \{(0,e), (3,e), (1,s), (2,s), (3,s \cdot t), (4,s \cdot t)\}\  \cup \  \{(2,u),(3,u) \, \vert \, u \notin \{s, s \cdot t, e\}\},$ 
\item[(b)] every element of 
\[ \{(4, s \cdot t)\} \cup \{(3,u) \, \vert \, u \notin \{s,t\} \} \subset {\rm Ap}(\mS,x), \]
is a maximal  with respect to $\leq_{\mS}$,
\item[(c)] ${\rm Ap}(\mS,y) =  \{(2,u) \, \vert \, u \notin \{s^2, s \cdot t\}\} \cup \{3, s^2 \cdot t\} \cup \{(0,e)\}, $
and, hence, it has size $|T|$, and
\item[(d)] $(0,e)$ and $(2, s^2 \cdot t \cdot s^{-1})$ belong to ${\rm Ap}(\mS,y)$ and are not maximal in this set with respect to $\leq_\mS$; thus
${\rm Ap}(\mS,y)$ has at most $|T|-2$ maximal elements.
\end{itemize}

If we prove facts (a)-(d), we conclude that ${\rm Ap}(\mS,x)$ has at least $|T| - 1$ maximal elements meanwhile ${\rm Ap}(\mS,y)$ has at most $|T|-2$; and the result follows. Since (a) is easy to check, we start by proving (b). We denote by $M_x$ the set of maximal elements of ${\rm Ap}(\mS,x)$ with respect to $\leq_{\mS}$. We first observe that $(4,s\cdot t) \in M_x$. An element $(3,u) \in {\rm Ap}(\mS,x)$ if and only if $u \neq s$; moreover it belongs to $M_x$ if and only if $(3,u) \not \leq_{\mS} (4,s \cdot t)$ or, equivalently, if $(1, s \cdot t \cdot u^{-1}) \not= (1, s)$;  this happens if and only if $u \neq t$. 

To prove (c) one just needs to verify that $(i,u) \in y \cdot \mS$ for all $u \geq 4$, that $(3,u) \in y \cdot \mS$ if and only if $u \neq s^2 \cdot t$, and that $(2,u) \in y \cdot \mS$ if and only if $u = s^2$.

 Clearly, $(0,e)$ is not maximal in ${\rm Ap}(\mS,y)$. Concerning $(2,s^2 \cdot t \cdot s^{-1})$, we observe that $s^2 \cdot t \cdot s^{-1} \neq s^2$ because $t \neq s$ and that $s^2 \cdot t \cdot s^{-1} \neq s \cdot t$ because $s \cdot t \neq t \cdot s$; hence $(2,s^2 \cdot t \cdot s^{-1}) \in {\rm Ap}(\mS,y)$. Finally, $(2,s^2 \cdot t \cdot s^{-1}) \leq_{\mS} (3, s^2 \cdot t)$ because $(1,s) \in T$. This proves (d) and we are done.
 \end{proof}


In order to generalize Theorem~\ref{simetricoBgana}, we first extend the definition of symmetric semigroup to numerical semigroups with torsion. Let $\mS$ be a numerical semigroup with torsion (since $\mS$ is abelian we will use the additive notation for the operation on $\mS$). We say that $\mS$ is {\it symmetric} if and only if there exists an $x \in \mathcal N(\mS)$ such that: 
$$y \in \mathcal N(\mS) \Longleftrightarrow x - y \in \mS.$$
We observe that the role of $x$ in this definition can only be played for a maximal element in $\mathcal N(\mS)$; indeed, $\mS$ is symmetric if and only if the set of gaps of $\mS$ has a maximum. Following the lines of the proof of Proposition~\ref{maxinvariantelattice}, we have that this is equivalent to ${\rm Ap}(\mS,z)$ having a global maximum with respect to $\leq_{\mS}$ for all $z \in \mS$. Thus, from Propositions~\ref{finiteafteronemove} and~\ref{maxinvariantelattice}, we conclude the following result:

\begin{theorem}\label{simetricoBganalattice} Let $T$ be a finite commutative group and let $(\mS,\leq_{\mS})$ be an ordered numerical semigroup with torsion.
If $\mS$ is symmetric, then player $B$ has a winning strategy for chomp on $\mS$.
\end{theorem}

Furthermore, we can extend the proof of Theorem~\ref{boundonstrategy} to prove the following result:

\begin{theorem}\label{boundonstrategylattice}
Let  $\mS \subset \N \times T$ be an ordered numerical semigroup with torsion. 
If $A$ has a winning strategy for chomp on $\mS$, then $A$ has a winning strategy with first move $(a,x)$ with $a \leq 2^{g t 2^{n}}$, where $g = \max \{a \in \Z \, \vert \, (a,x) \in \mathcal N(\mS)\}$, $n = |\mathcal N(\mS)|$ and $t = |T|$.
\end{theorem}

\section{Concluding remarks}\label{sec:conclude}
We have combined the combinatorial game of chomp with the more algebraically structured posets coming from numerical semigroups. Some concepts in semigroups found a corresponding interpretation as strategies in chomp. While we have found strategies on several classes of semigroups, it is clear, that the problem in general is quite complicated and for many (seemingly restricted) classes, we could not provide a complete characterization for when $A$ wins. We believe that the possibly easiest next case to attack would be to get more results on semigroups, that are generated by an interval. As mentioned at the end of Section~\ref{sec:intervals} a first task could be to completely characterize the  generated semigroups of type $2$ and with $k$ odd, where $A$ has a winning strategy. Another problem pointed out at the end of Section~\ref{sec:intervals} can be formulated as follows:

\begin{conjecture}\label{conjecture}
 For every $c$ there exists an $a_c$, such that for all $a\geq a_c$ player $A$ has a winning strategy for chomp on $\mS = \langle a, a+1,\ldots, 2a-c\rangle$ if and only if $a$ and $c$ are odd. 
\end{conjecture}
The results of Sections~\ref{sec:max} and~\ref{sec:intervals} show that the statement of the conjecture is true for $1\leq c\leq 3$, with the $a_1, a_2, a_3$ being $2,3,7$, respectively. Moreover, the results from Section~\ref{sec:arithmetic} show that if $a$ and $c$ are odd, then $A$ has a winning strategy, so indeed only one direction remains open. Similarly, one can wonder whether for every $k$ there exists an $a_k$, such that for all $a\geq a_k$ the winner of chomp on $\mS = \langle a, a+1,\ldots, a+k\rangle$ is determined by the parities of $a$ and $k$. We know this to be true for $1\leq k\leq 2$ by the results from Section~\ref{sec:intervals}.

%
%

We have shown, that the outcome of chomp on a given $\mS$ is decidable, by giving a huge upper bound for the smallest winning first move of $A$. Can this bound be improved? What is the computational complexity of deciding the outcome of $\mS$?

In the last section, we studied under which circumstances our results on numerical semigroups can be generalized in a straight-forward way to other posets coming from infinite monoids. It would be interesting to further broaden this class of posets, to those defined in terms of wider classes of monoids or satisfying some local symmetry conditions, such as the auto-equivalent posets defined in~\cite{cgpr}. This will be the subject of a forthcoming paper~\cite{averpacuando}.

\subsubsection*{Acknowledgements} 
K.K. has been supported by ANR project ANR-16-CE40-0009-01: Graphes, Algorithmes et TOpologie.   I. G. has been supported by Ministerio de
Econom\'ia y Competitividad,  Spain (MTM2016-78881-P).  

The authors want to thank the anonymous referees for their comments and suggestions that helped to improve this manuscript.

\bibliographystyle{plain}
\bibliography{chomp-bib}
\end{document}